\documentclass[USenglish]{article}
\usepackage{fullpage}
\usepackage{amsmath,amsfonts,amsthm,amssymb}
\usepackage{color,xcolor}
\usepackage{ifpdf}
\usepackage{psfrag}
\usepackage{graphicx,graphics}
\usepackage{hhtensor}
\usepackage{comment}
\usepackage[small]{caption}
\usepackage{subcaption}

\usepackage{xparse}
\usepackage{bigints}

\usepackage{url}
\usepackage{hyperref}
\usepackage{todonotes}
\usepackage{ulem} 
\newtheorem{theorem}{Theorem}[section]

\newtheorem{proposition}[theorem]{Proposition}
\newtheorem{remark}[theorem]{Remark}

\newtheorem{example}[theorem]{Example}

\newtheorem{conclusion}[theorem]{Conclusion}

\usepackage{cancel}
\usepackage{algorithm,algorithmic}
\usepackage{enumitem}

\newcommand{\R}{\mathbb R}

\newcommand{\diff}[1]{{\mathrm{d}{#1}}}

\definecolor{darkgreen}{rgb}{0.0, 0.42, 0.24}

\newcommand{\bu}{\mathbf{u}}

\newcommand{\bbu}{\mathbf{u}}

\newcommand{\ww}[1]{\underline{#1}}
\renewcommand{\div}{\operatorname{div}}

\newcommand{\dd}{\mathrm{d}}

\newcommand\norm[1]{\left\lVert#1\right\rVert}

\newcommand{\cc}[1]{\chi_{[t^0,t^{#1}]}}

\newcommand{\tp}{t^{n+1}}
\newcommand{\tn}{t^{n}}

\renewcommand{\vec}[1]{\ww{#1}}
\NewDocumentCommand{\mat}{mo}{%
  \IfValueTF{#2}{%
    \textrm{#1}{#2}
  }{%
    \textrm{#1}\,
  }%
}

\usepackage{bbm}
\def\L{\mathcal{L}}
\def\I{\mathcal{I}}
\def\R{\mathbb{R}}
\def\bbc{\underline{\boldsymbol{\alpha}}}
\def\bc{\boldsymbol{\alpha}}
\def\bbu{\underline{\boldsymbol{u}}}
\def\bu{\boldsymbol{u}}
\def\br{\boldsymbol{r}}
\def\bd{\mathbf{d}}
\def\bbd{\underline{\mathbf{d}}}
\def\bphi{\underline{\phi}}
\def\M{\underline{\underline{\mathrm{M}}}}
\def\S{\underline{\underline{\mathrm{S}}}}

\definecolor{darkspringgreen}{rgb}{0., 0.55, 0.3}
\definecolor{dartmouthgreen}{rgb}{0.05, 0.5, 0.06}
\definecolor{etonblue}{rgb}{0.59, 0.78, 0.64}
\definecolor{airforceblue}{rgb}{0., 0.4, 0.66}
\definecolor{arylideyellow}{rgb}{0.91, 0.84, 0.42}
\definecolor{emerald}{rgb}{0.31, 0.78, 0.47}
\definecolor{uclagold}{rgb}{1.0, 0.7, 0.0}
\definecolor{cadmiumorange}{rgb}{0.93, 0.53, 0.18}

\hypersetup{
pdftitle={}
pdfauthor={M. Han Veiga, P. \"Offner and D. Torlo},
pdfpagemode=UseOutlines,
linkbordercolor=0 0 0,
linkcolor=red,
citecolor=blue,
colorlinks=true,
bookmarks = true
}
\begin{document}
\title{DeC and ADER: Similarities, Differences and a Unified Framework}

\author{Maria Han Veiga\thanks{Michigan Institute of Data Science, University of Michigan, USA}, and Philipp \"Offner\thanks{Institute of Mathematics, Johannes Gutenberg-University, Germany}, and Davide Torlo\thanks{Inria Bordeaux Sud-Ouest, - 200 av.  de la vieille tour, 33405 Talence, France}
}
\date{}
\maketitle

\begin{abstract}
In this paper, we demonstrate that the explicit ADER approach as it is used \textit{inter alia} in \cite{zanotti2015space} 
can be seen as a special interpretation of the deferred correction (DeC) method 
as introduced in \cite{dutt2000dec}.
By using this fact, we are able to embed ADER in a theoretical background of time integration schemes and prove the relation between the accuracy order and the number of iterations which are needed to reach the desired order.
Next, we extend our investigation to stiff ODEs, treating these source terms implicitly. Some differences in the interpretation and implementation can be found. Using DeC yields typically a much simpler implementation, while ADER benefits from a higher accuracy, at least for our numerical simulations.  Then, we also focus on the PDE case and present common space-time discretizations using DeC and ADER in closed forms.
Finally, in the numerical section we investigate A-stability for the ADER approach - this is done for the first time up to our knowledge - for different order using several basis functions and compare them with the DeC ansatz. Then, we compare the performance of ADER and DeC for stiff and non-stiff ODEs and verify our analysis focusing on two basic hyperbolic problems.

\end{abstract}
\section{Introduction}\label{sec:intro}
 
Very high-order methods have become rather ubiquitous in the field of numerical methods for hyperbolic partial differential equations. 
There are many methods which obtain an arbitrarily high-order in its spatial discretization, namely, discontinuous Galerkin \cite{glaubitz2020stable,offner2019error}, spectral difference method \cite{Liu2006,glaubitz2018application}, etc. 
When considering time-dependent problems, the time integration must have the same order of convergence as of the spatial one, in order to formally guarantee the high order space-time convergence. Explicit Runge--Kutta methods have long been used for its simplicity and ease of implementation, however, they are difficult to generalize to very high order because there is no automatic procedure to generate the Butcher tableau. There are explicit timestepping methods which promise arbitrarily high accuracy, without the necessity to compute Butcher like coefficient tables, such as deferred correction \cite{dutt2000dec}, ADER \cite{toro2001towards},
SBP \cite{nordstrom2013summation} or (continuous or discontinuous) Galerkin approaches also in time.

The deferred correction (DeC) method was first introduced in the context of ODEs, and later, formulated as a timestepping scheme for PDEs in conjunction with finite element (FE) methods \cite{abgrall2017dec,liu2008strong}.
The key idea of DeC is based on the Picard-Lindel\"of Theorem  and  especially on the Picard iteration.
In every iteration step, we decrease the error of the numerical solution until we reach a fixed  bound.
Extension to implicit or semi-implicit variations exist \cite{minion2003dec}
but we concentrate mainly on the explicit version. From our point of view,  some advantages of DeC compared 
to explicit RK are the possibility to use an FE ansatz in space and avoid the inversion of the accompanying mass matrix to obtain a full discretization of a PDE as described in \cite{abgrall2017dec} or to obtain arbitrarily high order time accuracy without computing the order conditions checks on the coefficients of the Butcher tableau. 

The ADER approach, introduced firstly in \cite{toro2001towards} and further developed in many other works, e.g. \cite{dumbser2008unified,titarev2002ader,zanotti2015space},
remains quite elusive and misunderstood in the broad community of numerical analysis for hyperbolic problems, despite being able to achieve arbitrarily high order in time and being  able to be even more efficient in terms of runtime than classical RK approaches \cite{balsara2013efficient}. 

The first ADER methods for linear hyperbolic equations were presented in \cite{Schwartzkopff2002ADER,toro2001towards}.
The (historical) ADER (Advection-Diffusion-Reaction) approach, described in the paper ``ADER: Arbitrary High Order Godunov Approach" \cite{titarev2002ader}, extends the method to nonlinear hyperbolic systems, achieving arbitrarily high order accuracy both in time and space. The key ingredient of this approach is to consider a generalized Riemann Problem \cite{toro2009riemann}. A comprehensive stability and truncation-error analysis for Finite--Volume ADER schemes can be found in \cite{TitarevWAF2006}.

Later, in \cite{dumbser2008unified}, a different formulation of ADER is presented, which can be interpreted as a space-time finite element method.  This is the ADER approach that we consider in this paper, referring to it as the \textit{modern ADER} (described in detail in sections \ref{sec:ader} and \ref{sec:beyondODE}). Although used abundantly in many codes, the theoretical properties of the (modern) ADER are somewhat lacking. Some literature shows, for example, that for a linear homogeneous system, a finite number of iterations is sufficient to convergence \cite{hjackson2017}. However, questions on how to choose integration points, polynomial bases, are only addressed empirically up-to-our knowledge. For a comprehensive description of the development of ADER schemes, from their conception to the modern version of ADER, please refer to \cite{Busto2020}.

In order to deal with stiff problems, implicit methods must be employed to guarantee stability, see \cite{hairer96ODE2,butcher08numODE} for reviews of implicit ODE solvers. Carrying on with the investigation of similarities between DeC and ADER, we study their implicit variants as proposed in \cite{abgrall2018asymptotic} and \cite{dumbser2007FVStiff}, respectively.

Since DeC and the modern ADER are both iterative approaches, the questions that motivated this paper are the following:
\begin{enumerate}
\item Is there a connection between the ADER timestepping method and the DeC timestepping method for explicit ODEs? And if yes, what is this connection?
\item Does the connection between ADER and DeC help us study properties of the ADER scheme?
\item How do these methods change when we consider stiff ODEs?
\end{enumerate}

In order to address these questions, we first introduce the two considered methods. In section \ref{sec:ader}, we describe the ADER approach and, 
in section \ref{sec:originalDeC}, the Deferred Correction method. Then, in section \ref{sec:Relation}, we show that these two methods are very similar and that DeC can be written as an ADER scheme and \emph{vice versa}. Furthermore, we study how these methods differ when considering stiff ODEs and PDEs. These connections and differences are then verified through numerical experiments shown in section \ref{sec:results}. Finally, we conclude the paper with a discussion in section \ref{se:Summary}.
\section{The (Modern) ADER Approach}
\label{sec:ader}

In the following section, we will give a short introduction about the ADER approach and how we understand it.
Actually, it is an iterative process to 
obtain the numerical solution 
of a given space--time PDE. 
Therefore, we consider for simplicity
the following scalar (non)linear hyperbolic problem 
\begin{equation}\label{eq:scalarPDE}
  \partial_t u + \partial_x f(u) = 0.
\end{equation}
Suppose that $u(x,t)$ solves this equation under sufficient boundary and initial conditions. The main idea of the modern ADER is based on the variational formulation in the finite element context. We approximate the field $u(x,t)$
in a set of space nodes $\{x_i\}_{i=1}^I \subset \mathbb{R}$ and we denote with the vector $\bc(t) = \lbrace u(x_i,t) \rbrace_{i=1}^I : \R^+\to \R^I$ the semi-discretization in space of the function $u$\footnote{By switching from the PDE formulation to the ODE setting, we change also the used notation. Here, $u$ is used in the PDE case and $\alpha$ when speaking about an ODE.}. 

For the time discretization, we consider a time interval 
$T^n := \left[t^n, t^{n+1} \right]$, 
and we represent $\bc(t)$ as a 
linear combination of basis functions in the time component $t$:
\begin{equation}\label{eq:basis_reconstruction}
\bc(t) = \sum_{m=0}^M \phi_m(t)\bc^m = \bphi(t)^T\bbc,
\end{equation}
where $\bphi = [ \phi_0, \dots , \phi_M ]^T:T^n\to \R^{M+1}$ is the vector of time basis functions and $\bbc=[ \bc^0, \dots, \bc^M]^T \in \R^{(M+1)\times I}$ 
is the vector of coefficients related to the time basis functions for each degree of freedom in space. Typically, $\bphi$ are Lagrange basis functions in some nodes 
\begin{equation}\label{eq:subtimesteps}
\lbrace t_{m} \rbrace_{m=0}^M \subset T^n,
\end{equation}
e.g. equispaced, Gauss-Lobatto or Gauss-Legendre nodes.

Proceeding with the weak formulation in time, 
we consider a smooth test function $\psi(t): \R \to \R$ 
and we integrate over some interval $T^n$. 

We consider, from now on, only the time derivatives of the hyperbolic equation \eqref{eq:scalarPDE}. In other words, we can say that, through the method of lines, we make a splitting in space and time because the ADER approach will be used as the time integration method and it will be fully explicit. 
For the space discretization, one can use their favorite numerical scheme, \textit{inter alia} discontinuous Galerkin (DG), flux reconstruction (FR), finite volume (FV) or ENO/WENO methods which yield the different ADER representations, like ADER--WENO or ADER--DG that can be found in literature. Additionally,  we will not suffer from any other issues like accuracy reduction or similar if using one of the above described space discretization methods.

We consider, thus, the function $F:\R^I\to \R^I$ to be the semi-discrete operator in space for all the degrees of freedom, which is given by the chosen spatial scheme. We can just focus then on the resulting system of ODEs for $\bc : [0,T]\to \R^I$
\begin{equation}\label{eq:scalarODE}
\partial_t \bc + F(\bc) =0.
\end{equation}

Its variational form in time is given by 
\begin{equation}\label{eq:ADERODEL2}
\int_{T^n}  \psi(t)\partial_t \bc(t) dt + \int_{T^n} \psi(t)F(\bc(t))  dt = 0, \quad  \forall \psi: T^n\to \R.
\end{equation}

Now, we replace the unknown with its reconstruction and we choose the test functions to be the same as the basis functions. This results in the following definition of the operator $\L^2: \R^{(M+1 )\times I} \to \R^{(M+1 )\times I} $
\begin{equation}\label{eq:L2_ADER}
\L^2(\bbc ):= \int_{T^n} \bphi(t) \partial_t \bphi(t)^T \bbc dt + \int_{T^n} \bphi(t)  F(\bphi(t)^T\bbc)  dt = 0.
\end{equation}

Integrating by parts yields
\begin{equation}\label{eq:L2_ADER_ibp}
\L^2(\bbc )= \bphi(t^{n+1}) \bphi(t^{n+1})^T \bbc
- \bphi(t^n) \bc(t^n) 
- \int_{T^n}  \partial_t\bphi(t) \bphi(t)^T dt  \bbc
+ \int_{T^n}\bphi(t) F(\bphi(t)^T\bbc ) dt = 0.
\end{equation}

Now, we replace the integrals by quadratures. The quadrature nodes can coincide (or not) with the ones defining the Lagrange polynomials\footnote{In this work we choose the quadrature nodes as Gauss-Lobatto or Gauss-Legendre nodes, according to the definition of the Lagrange polynomials. However, in many application of ADER, Gauss-Legendre nodes are the typical choice to guarantee the exactness 
of the integration.}. We denote them as $\lbrace t^q_z \rbrace_{z=0}^Z \subset T^n$ and
 with respective weights $\lbrace w_z \rbrace_{z=0}^Z$. We choose the weights, the nodes and the basis functions to be scaled for the interval $[0,1]$, so that we explicitly highlight the presence of $\Delta t=t^{n+1}-t^n$. The integration terms reduce to 
 
 \begin{equation}\label{eq:quadrature}
\begin{aligned}
  \int_{T^n} \partial_t \bphi(t) \bphi(t)^T dt \bbc &
\approx \sum_{z=0}^Z w_z \partial_t \bphi(t^q_z) \bphi(t^q_z)^T\bbc  \\
\int_{T^n}  \bphi(t)F(\bphi(t)^T\bbc)  dt &\approx
\Delta t \sum_{z=0}^{Z} w_z \bphi(t^q_z) F(\bphi(t^q_z)^T\bbc).
\end{aligned}
\end{equation}

Then, the approximation of \eqref{eq:L2_ADER_ibp} yields, for every test function indexed by $m=0,\dots, M$
\begin{equation}\label{eq:System}
 \phi_m(t^{n+1})\bphi(t^{n+1})^T\bbc -
\sum_{z=0}^Z w_z \partial_t \bphi(t^q_z) \bphi(t^q_z)^T\bbc
=\phi_m(t^{n})\bc(t^n) 
- \Delta t\sum_{z=0}^{Z} w_z \bphi(t^q_z) F(\bphi(t^q_z)^T\bbc) .
\end{equation}

This is a system of $M+1$ equations with $M+1$ unknowns for every ODE of the system \eqref{eq:scalarODE}. 
Actually, what is expressed here is nothing more than a classical collocation method, c.f. \cite{wanner1996solving}, or a high order implicit RK method.
The order depends on the used quadrature rule. 
Using, as an example, Gauss--Legendre or Gauss--Lobatto nodes both for the definition of the basis functions and the quadrature nodes results in a high order quadrature formula. One can also use different points for the quadrature and the basis functions, resulting in more varieties of this scheme. 

We can rewrite the system \eqref{eq:System} in a matrix--fashioned way, using the mass matrix $\M \in \R^{(M+1)\times(M+1)}$, given by
\begin{equation}\label{eq:MassmatrixAder}
\M_{m,l} := \phi_m(t^{n+1})\phi_l(t^{n+1})- \sum_{z=0}^Z \partial_t \phi_m(t^q_z) \phi_l(t^q_z) w_z
\end{equation}
and the right--hand side functional $\vec{r}(\bbc):\R^{(M+1)\times I}\to \R^{(M+1)\times I}$, given by

\begin{equation}\label{eq:aderrhs}
\vec{r}(\bbc)_m:=\phi_m(t^{n})\bc(t^n)
- \Delta t\sum_{z=0}^{Z} w_z \phi_m(t^q_z) F(\bphi(t^q_z)^T\bbc), \quad m=0,\dots,M+1.
\end{equation}

We have brought on the right--hand  side  of \eqref{eq:System} the nonlinear terms and the explicit ones, while keeping the linear terms on the left--hand side.
Finally, our system to be solved is given by

\begin{equation}\label{fix:point}
 \M \bbc = \vec{r}(\bbc) \Longleftrightarrow \L^2(\bbc):=\M\bbc-\vec{r}(\bbc)\stackrel{!}{=}0.
\end{equation}

This equation is nothing else than a fixed--point problem. Its solution will give us an $(M+1)$th order accurate solution in time. It cannot be directly solved when nonlinear fluxes are present, but it can be solved 
under certain assumptions with an iterative process in $\bbc$.
Defining the starting guess as $\bbc^{(0)}:=[\bc(t^n),\dots, \bc(t^n) ]^T$, the algorithm proceeds iteratively as follows
\begin{equation}\label{eq:fixpoint_iteration}
   \bbc^{(k)}=\M^{-1} \vec{r}(\bbc^{(k-1)}), \quad k=1,\dots,K.
\end{equation}

Several questions arise automatically for \eqref{eq:fixpoint_iteration},
such as, how is the convergence of the method influenced by the mass matrix $\M$ and, hence, by the node placement?
By re-interpreting the new ADER approach into the DeC framework in section \ref{sec:Relation}, we can answer these and more questions, thanks to the definition of the $\L^2$ operator in \eqref{fix:point}. This serves as a connection to the DeC procedure. 

Before introducing the DeC method, we give the following simple example to get more familiar with the ADER method in the ODE setting.

\begin{example}\label{ex:ADERexample}{\textbf{2$^{nd}$ order ADER method for ODEs}}
Let us demonstrate a concrete example of the methodology described above, considering a simple scalar ODE:

\begin{equation}
      \alpha'(t) + F(\alpha) = 0, 
\end{equation}
with $\alpha:[0,T] \to \mathbb{R}$.

Let us consider the timestep interval $[t^n,t^{n+1}]$, rescaled to $[0,1]$. The time interpolation nodes and the quadrature nodes are given by Gauss-Legendre points (in the interval $[0,1]$) and respective quadrature weights.
\[\ww{t}_q =  \left( t^0_q, t^1_q \right) = \left( t^0, t^1 \right) =  \left( \frac{\sqrt{3}-1}{2\sqrt{3}}, \frac{\sqrt{3}+1}{2\sqrt{3}} \right), \quad \ww{w} = \left(1/2,1/2\right). \]

The time basis is given by Lagrange interpolation polynomials built on the nodes $\vec{t}^q$. 

\[\ww{\phi}(t) = \left( \phi_0(t), \phi_1(t) \right) = \left( \frac{t-t^1}{t^0-t^1}, \frac{t-t^0}{t^1-t^0} \right). \]

Then, the mass matrix is given by
\[\M_{m,l} = \phi_m(1)\phi_l(1) - \phi'_m(t^l) w_l, \quad m, l = 0,1,\]
thanks to the definition of the Lagrange polynomials.
\[ \M =
\begin{pmatrix}
   1   & \frac{\sqrt{3}-1}{2}  \\
    -\frac{\sqrt{3}+1}{2}   & 1
\end{pmatrix}.\]

The right hand side is given by\footnote{Note that the quadrature process is greatly simplified because of the choice of Gauss--Legendre nodes both for the quadrature and the Lagrangian basis functions in \eqref{eq:MassmatrixAder} and \eqref{eq:aderrhs}.}

\[ r(\bbc)_m = \alpha(0) \phi_m(0) - \Delta t F(\alpha(t^m)) w_m, \quad m=0,1. \] 

\[ \vec{r}(\bbc) = \alpha(0)\vec{\phi}(0) -\Delta t
\begin{pmatrix}
    F(\alpha(t^1)) w_1  \\
    F(\alpha(t^2)) w_2.
\end{pmatrix}.\]

Then, the coefficients $\bbc$ are given by

\begin{align*}
\bbc^{(k+1)} &= \M^{-1} \vec{r}( \bbc^{(k)} ).
\end{align*}
Finally, use $\bbc^{(k+1)}$ to reconstruct the solution at the time step $t^{n+1}$:

\begin{align*}
\alpha^{n+1} &= \vec{\phi}(1)^T \bbc^{(k+1)}.
\end{align*}

\end{example}

\subsection{Historical ADER}
In the community, the term ADER is often associated with the historical approach. In order to establish the difference between the modern version of ADER and the historical one, we explain the procedure, as introduced in \cite{titarev2002ader}, highlighting the key differences between the two methods.

Solving the same problem as in the PDE \eqref{eq:scalarPDE}, let us consider $t^n$ as starting time and 
suppose that we have at our disposal the cell averages: $u_i(t^n)$, for $i$ indexing control volumes $V$.

The method consists of three steps:

\begin{enumerate}
\item{reconstruction of point-wise values from cell averages (using some reconstruction function $\mathcal{R}$)
\[ \mathcal{R}(u_{i-k},...,u_{i+l},t^n) = u(x,t^n), \quad x \in V_i, \]
}
\item{solution of the generalized Riemann problem at the cell interfaces,}
\item{evaluation of the intercell flux to be used in the conservative scheme.}
\end{enumerate} 

The main difference is in step two. In order to solve the generalized Riemann problem at the cell interfaces (for all local times $\tau:=t-t^n$), one writes the Taylor expansion of the interface state in time

\begin{equation}
\label{eq:taylorexpADER}
 u(x_{i+1/2},\tau) = u(x_{i+1/2},0^+) + \sum^{m}_{k=1} \left[ \partial^{(k)}_t u(x_{i+1/2}, 0^{+})\right]\frac{\tau^k}{k!}
\end{equation}

Note that the leading term, $u(x_{i+1/2},0^+)$, accounts for interactions with the boundary extrapolated values $u_L(x_{i+1/2})$ and $u_R(x_{i+1/2})$, and it is the Godunov state of the conventional (piece--wise constant data) Riemann problem. Typically, an exact or approximate Riemann solver is used to provide the first term of the approximation. 

The next terms are higher order corrections to the 0th Godunov state. The high order derivatives in time are replaced by derivatives in space by means of the Cauchy-Kovalevskaya procedure.

As in \cite{grptoro}, the space derivatives $ \partial ^{(k)}_x u(x,t)$ of the solution at $(x-x_{i+1/2},\tau) = 0 $ can be evaluated as the Godunov states of the following linearized generalized Riemann Problem:

\begin{equation}\label{eq:RP}
\begin{aligned}
&\partial_t u^{(k)} + A(u(x_{i+1/2},0^+))\partial_x  u^{(k)} = 0, \\
&u^{(k)}(x,0) = \begin{cases}
\partial_x^{(k)} u_L(x_{i+1/2}), \quad x<x_{i+1/2},\\
\partial_x^{(k)} u_R(x_{i+1/2}), \quad x>x_{i+1/2}.
\end{cases}
\end{aligned}
\end{equation}

The initial condition for the Riemann Problem \eqref{eq:RP} above is given by differentiating the high--order reconstruction polynomial with respect to $x$. Having all derivatives in terms of their spatial component, the Taylor expansion \eqref{eq:taylorexpADER} can be evaluated as
\begin{equation}
u(x_{i+1/2},\tau) = a_0 + a_1 \tau + a_2 \tau^2 + ... + a_M \tau^M
\end{equation}
for $a_j$ containing all the constants not depending on $\tau$. This expression approximates the interface state for $0\leq \tau \leq \Delta t$ to $(M+1)$th order of accuracy.

Finally, to evaluate the numerical flux (now in time as well), an appropriate Gaussian rule is used:
\begin{equation*} \hat{F}_{i+1/2} = \sum_{q=0}^{M} F(u(x_{i+1/2}, t^q \Delta t )) \omega_q,
\end{equation*}
where $t^q$, $\omega_q$ are nodes and weights of the quadrature rule, and $M+1$ the number of nodes.
\section{Deferred Correction Methods}\label{sec:originalDeC}

In this section, we focus on the \textbf{deferred correction (DeC)} method which was introduced by Dutt et al.  \cite{dutt2000dec} and then reinterpreted by Abgrall \cite{abgrall2017dec}. It is an explicit, arbitrarily high order method for ODEs  but further extensions of DeC, including implicit, semi-implicit and modified Patankar versions,  can be found  nowadays in the literature \cite{christlieb2010integral,minion2003dec,offner2019arbitrary}.

Since we want to embed the modern ADER approach into this framework, we focus only on the explicit version. Therefore, we use the  notation of DeC introduced by Abgrall in \cite{abgrall2017dec} because, in our opinion, it is more convenient to prove accuracy than in previous works   \cite{dutt2000dec,christlieb2010integral,liu2008strong}.
Actually, Abgrall focuses on DeC as a time integration scheme in the context of finite element methods. In particular, when using continuous Galerkin schemes for the space discretization and applying RK methods, a sparse mass matrix has to be inverted. Through the DeC approach, one can avoid the mass matrix inversion. 

The main core of all the DeC algorithms is the same 
and it is based on the Picard-Lindel\"of theorem in the continuous setting. The theorem states the existence and uniqueness of solutions for ODEs. The classical proof makes use of the so--called Picard iterations to minimize the error and to prove the convergence of the method, which is nothing else than a fixed--point problem. The foundation of DeC relies on mimicking the Picard iterations and the fixed--point iteration process at the discrete level.

Here, we see already some connection between the two approaches. Indeed, the approximation error decreases with several iteration steps. For the description of DeC,
let us consider the system of ODEs as in \eqref{eq:scalarODE}
\begin{equation*}
\bc'(t)+F(\bc)=0,
\end{equation*}
with $\bc:[0,T]\to \R^I$.
 Abgrall
introduces two operators: $\L^1$ and $\L^2$.
The $\L^1$ operator represents a low-order
easy--to--solve numerical scheme, e.g. the explicit Euler method, 
and $\L^2$ is a high-order operator
that can present difficulties in its practical solution, e.g. an implicit RK scheme or a collocation method.
We use here on purpose the $\L^1, \L^2$ nomenclature since we will see that actually 
this is the key point in the connection. The DeC method can be written as a combination of these two operators.

Given a timeinterval $[t^n, t^{n+1}]$ we subdivide  
it into $M$ subintervals  $\lbrace [t^{n,m-1},t^{n,m}]\rbrace_{m=1}^M$,
where $t^{n,0} = t^n$ and $t^{n,M} = t^{n+1}$
and we mimic for every 
subinterval $[t^0, t^m]$ the Picard--Lindel\"of theorem 
for both operators $\L^1$ and $\L^2$. 
We drop the dependency on the timestep $n$ for subtimesteps $t^{n,m}$ 
and substates $\bc^{n,m}$ as denoted in Figure \ref{Fig:Time_interval}.  
\begin{figure}[ht]
	\centering
\begin{tikzpicture}
\draw [thick]   (0,0) -- (10,0) node [right=2mm]{};
\fill[black]    (0,0) circle (1mm) node[below=2mm] {$t^n=t^{n,0}=t^0 \,\, \quad$} node[above=2mm] {$\bc^0$}
                            (2,0) circle (0.7mm) node[below=2mm] {$t^{n,1}=t^1$} node[above=2mm] {$\bc^1$}
                            (4,0) circle (0.7mm) node[below=2mm] {}
                            (6,0) circle (0.7mm) node[below=2mm] {$t^{n,m}=t^m$} node[above=2mm] {$\bc^m$}
                            (8,0) circle (0.7mm) node[below=2mm] {}
                            (10,0) circle (1mm) node[below=2mm] {$\qquad t^{n,M}=t^M=t^{n+1}$} node[above=2mm] {$\bc^M$}; 
\end{tikzpicture} \caption{Subtime intervals}\label{Fig:Time_interval}
\end{figure}
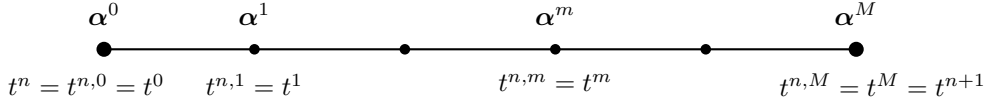

Then, the $\L^2$ operator is given by
\begin{equation}\label{eq:L2operator}
\L^2(\bc^0, \dots, \bc^M) :=
\begin{cases}
 \bc^M-\bc^0 -\int_{t^0}^{t^M} \I_M ( F(\bc^0),\dots,F(\bc^M))
= \bc^M-\bc^0 -\sum_{r=0}^M \int_{t^0}^{t^M} F(\bc^r) \phi_r(s) \diff s\\
\vdots\\
\bc^1-\bc^0 - \int_{t^0}^{t^1} \I_M ( F(\bc^0),\dots,F(\bc^M))
=\bc^1-\bc^0 - \sum_{r=0}^M \int_{t^0}^{t^1} F(\bc^r) \phi_r(s) \diff s
\end{cases}.
\end{equation}
Here, the term $\I_M$  denotes an interpolation polynomial of degree 
$M$ evaluated at the points $\lbrace  t^{r}\rbrace _{r=0}^M$.
In particular,
we use Lagrange polynomials $\lbrace \phi_r \rbrace_{r=0}^M$, 
where $\phi_r(t^{m})=\delta_{r,m}$ and $\sum_{r=0}^M \phi_r(s) \equiv 1$ for any $s\in [0,1]$. 

Using these properties, 
we can actually compute the integral of the interpolants, 
thanks to a quadrature rule in the same points $\lbrace t^m \rbrace_{m=0}^M$ 
with weights $\theta_r^m := \int_{t^n}^{t^{n,m}} \phi_r(s) ds$. 

We can rewrite
\begin{equation}\label{eq:L2}
\L^2(\bc^0, \dots, \bc^M) =
\begin{cases}
\bc^M-\bc^0 - \Delta t\sum_{r=0}^M \theta_r^M F(\bc^r)\\
\vdots\\
\bc^1-\bc^0 - \Delta t\sum_{r=0}^M \theta_r^1 F(\bc^r)
\end{cases}.
\end{equation}
The $\L^2$ operator represents an $(M+1)$th order numerical scheme (collocation method) if set equal
to zero, i.e., $\L^2(\bc^0, \dots, \bc^M)=0$. 
Unfortunately, the resulting scheme is implicit and, 
further, the terms $F$ may be nonlinear.
Because of this, the only $\L^2=0$ formulation 
is not explicit and more efforts have to be  made to solve it.

For this purpose, we introduce a simplification of  the $\L^2$ operator. 
Instead of using a quadrature formula at the points $\lbrace t^m \rbrace_{m=0}^M$ we evaluate 
the integral in equation \eqref{eq:L2operator} applying the left Riemann sum. 
The resulting operator $\L^1$ is given by the forward Euler discretization for each state $\bc^m$ in the timeinterval, i.e.,  
\begin{equation}\label{eq:L1}
\L^1(\bc^0, \dots, \bc^M) :=
\begin{cases}
 \bc^M-\bc^0 - \beta^M \Delta t F(\bc^0) \\
\vdots\\
\bc^1- \bc^0 - \beta^1 \Delta t F(\bc^0)
\end{cases}.
\end{equation}
with coefficients $\beta^m:=\frac{t^m-t^0}{t^M-t^0}$.\\
To simplify the notation and to describe  
DeC, as before, we introduce the vector of states for the variable $\bc$ at all subtimesteps

\begin{align}\label{eq:definition_bbc}
&\bbc :=  (\bc^0, \dots, \bc ^M) \in \R^{M\times I}, \text{ such that }\\
&\L^1(\bbc) := \L^1(\bc^0, \dots, \bc^M) \text{ and } \L^2(\bbc) := \L^2(\bc^0, \dots, \bc^M) .
\end{align}
Now, the DeC algorithm uses a combination of the $\L^1$ and $\L^2$ operators
to provide an iterative procedure.
The aim is to recursively approximate $\bbc^*$, the numerical solution of
the $\L^2(\bbc^*)=0$ scheme, similarly to the Picard iterations in the 
continuous setting. The successive states of the iteration process will be denoted 
by the superscript $(k)$, where $k$ is the iteration index, e.g. $\bbc^{(k)}\in \R^{M\times I}$.
The total number of iterations (also called correction steps) is denoted by $K$.
To describe the procedure, we have to refer to both 
 the $m$-th subtimestep and the $k$-th iteration of the DeC algorithm. We will indicate the variable by $\bc^{m,(k)} \in \R^I$.

Finally, the DeC method can be written as \\

 \centerline{\textbf{DeC Algorithm}}
\begin{equation}\label{DeC_method}
\begin{split}
&\bc^{0,(k)}:=\bc(t^n), \quad k=0,\dots, K,\\
&\bc^{m,(0)}:=\bc(t^n),\quad m=1,\dots, M\\
&\L^1(\bbc^{(k)})=\L^1(\bbc^{(k-1)})-\L^2(\bbc^{(k-1)}) \quad \text{ for }k=1,\dots,K.
\end{split}
\end{equation}
Using the procedure \eqref{DeC_method}, we  need, in particular, as many iterations as the desired  order of accuracy, i.e., $K= M+1$. 

Notice that, in every step, we solve the equations for the unknown variables $\bbc^{(k)}$ which appears only in the $\L^1$ formulation, the operator that can be easily inverted. Conversely, $\L^2$ is only applied to already computed predictions of the solution $\bbc^{(k-1)}$.
Therefore, the scheme \eqref{DeC_method} is completely explicit and arbitrary high order as stated in \cite{abgrall2017dec} with
 the following proposition.
\begin{proposition}\label{DeC_prop}
Let $\L^1$ and $\L^2$ be two operators defined on $\mathbb{R}^{M\times I}$, 
which depend on the discretization scale $\Delta = \Delta t$, such that
\begin{itemize}
\item $\L^1$ is coercive with respect to a norm, i.e., 
$\exists\, \gamma_1 >0$ independent of $\Delta$, such that for any $\bbc,\bbd$
we have that $$\gamma_1||\bbc-\bbd||\leq ||\L^1 (\bbc)-\L^1 (\bbd)||,$$
\item $\L^1 - \L^2$ is Lipschitz with constant $\gamma_2>0$
uniformly with respect to $\Delta$, i.e., for any $\bbc,\bbd$
$$
||(\L^1(\bbc)-\L^2(\bbc))-(\L^1(\bbd)-\L^2(\bbd))||\leq \gamma_2 \Delta ||\bbc-\bbd||.
$$
\end{itemize}
We also assume that there exists a
unique $\bbc^*_\Delta$ such that $\L^2(\bbc^*_\Delta)=0$. 
Then, if $\eta:=\frac{\gamma_2}{\gamma_1}\Delta<1$,
the DeC is converging to $\bbc^*$ and after $k$ iterations
the error $||\bbc^{(k)}-\bbc^*||$ is smaller than $\eta^k||\bbc^{(0)}-\bbc^{*}||$.
\end{proposition}

Proofs of this proposition and of the hypotheses of the proposition for operators $\L^1$ and $\L^2$ such as \eqref{eq:L1} and \eqref{eq:L2} can be found in \cite{abgrall2017dec,abgrall2018asymptotic,offner2019arbitrary}.

The condition for $\eta$ comes actually from the fixed--point theorem and has to be guaranteed
that the iterative process converges.
Now, before we focus on the relation between DeC and ADER, we give the following remark.

\begin{remark}

\begin{itemize}
\item In the $\L^1$ operator, one can use higher order time integration 
methods than the explicit Euler method. In principle,
this increases the convergence rate in the iteration 
process. However, an accuracy drop down can be 
observed in some cases. This is due to the fact
that the smoothness of the error is not guaranteed anymore,
c.f. \cite{christlieb2010integral}.

\item 
In our description of DeC both endpoints are included.
However, Gauss-Legendre nodes (like in example \ref{ex:ADERexample})
can be used. In this case, the approximation at the endpoint is done via extrapolation, see \cite{dutt2000dec}.
\item 
Finally, any DeC method can be interpreted as a RK scheme  \cite{christlieb2010integral}. 
The main difference between RK and  DeC is that the latter 
gives a general approach to the time discretization and does
not require a specification of the coefficients for every order of accuracy. 
\end{itemize}

\end{remark}

\section{Relation between DeC and ADER}\label{sec:Relation}

What we saw up to now was a repetition of the DeC and ADER approaches.  Both are based on an iterative procedure and mainly the foundation is given by the same fixed--point iteration method. 
In the following, we point out the relation between these two methods, namely, we show how ADER can be expressed as DeC and \textit{vice versa}. 
This relation can be used to prove  a new theoretical result for the ADER algorithm. 
The number of iterations needed to the Picard process can be chosen equal to the accuracy order we aim to reach. 
This result can be used in a general setting, providing few hypotheses on the operators, extending the result of \cite{hjackson2017}.

\subsection{ADER as DeC} 
First, we start to show how the modern ADER can be put into the 
DeC framework which we have described in section 
 \ref{sec:originalDeC}.
Therefore, let us rewrite
the $\L^2$ operator from section \ref{sec:ader}.
It is given in equation \eqref{eq:L2_ADER_ibp} and \eqref{fix:point}.
It is 
\begin{equation*}
 \L^2(\bbc):=\M\bbc-\vec{r}(\bbc),
\end{equation*}
where $\M$ is the previously defined invertible mass matrix. 
For the DeC algorithm \eqref{DeC_method},
we need further a low--order explicit operator. 
For the ADER, we choose the same low order operator of the DeC, namely, 
\begin{equation}\label{eq:L1ADER}
 \L^1(\bbc):= \M\bbc-\vec{r}(\bc(t^n)).
\end{equation}
Actually, we have to mention that this operator is not really unique and can be defined in different ways, since all the information is already included in the $\L^2$ operator for the ADER--DeC approach (see remark \ref{rem:l1}). Nevertheless, the choice of \eqref{eq:L1ADER} is useful for the hypotheses of the proposition \ref{DeC_prop}, because, in this way, the difference of the two operators will be Lipschitz continuous.

Then, we obtain the ADER-DeC algorithm:
\begin{equation*}
 \L^1(\bbc^{(k)})= \L^1(\bbc^{(k-1)})-\L^2(\bbc^{(k-1)}),\qquad k=1,\dots, K,
\end{equation*}
defining $\bc^{(k),0}=\bc(t^n)$, $\forall k $. Hence, we can explicitly write it as 
\begin{align*}
 &  \M \bbc^{(k+1)} -\vec{r}(\bc^{(k+1)}(t^n))- \M \bbc^{(k)} +\vec{r}(\bc^{(k)}(t^n))
   +\M \bbc^{(k)} -\vec{r}(\bbc^{(k)}) =0\\
&    \M \bbc^{(k+1)} -\vec{r}(\bbc^{(k)}) =0.
\end{align*}
which is nothing more than the discrete fixed--point problem in equation \eqref{fix:point}.

\begin{remark}
\label{rem:l1}
The operator $\L^2$ already comprises the information needed in the fixed--point iteration. Nevertheless, the $\L^1$ operator serves us to easily prove the convergence up to the required order of accuracy of the process.

Finally, we like to mention that one can also define the operators as
 \begin{equation}
    \begin{cases}
        \L^1(\bbc):= \bbc -\M ^{-1}\vec{r}(\bc(t^n)),\\
        \L^2(\bbc):= \bbc -\M ^{-1}\vec{r}(\bbc).
    \end{cases}
\end{equation} 
and obtain an analog result. We use this expression when demonstrating the accuracy property in subsection \ref{subsec:accuracy}. 
\end{remark}

\subsection{DeC as ADER}\label{eq:DeC_ader}

Here, we will show that the DeC scheme can be expressed as an ADER method. The essential difference is the choice of appropriate basis functions. This is mainly related to the definition of the $\L^2$ operator in the DeC framework.
If we rewrite  \eqref{eq:L2operator}
\begin{equation*}\label{eq:L2op}
\L^2(\bbc) :=
\begin{cases}
 \bc^M-\bc^0 -\int_{t^0}^{t^M} \I_M ( F(\bc^0),\dots,F(\bc^M))
= \bc^M-\bc^0 -\sum_{r=0}^M \int_{t^0}^{t^M} F(\bc^r) \phi_r(t) \diff t\\
\vdots\\
\bc^1-\bc^0 - \int_{t^0}^{t^1} \I_M ( F(\bc^0),\dots,F(\bc^M))
=\bc^1-\bc^0 - \sum_{r=0}^M \int_{t^0}^{t^1} F(\bc^r) \phi_r(t) \diff t
\end{cases}.
\end{equation*}
and focus on the $m$-th line, which reads
\begin{equation*}
 \bc^m-\bc^0-\sum_{r=0}^M F(\bc^r) \int_{t^0}^{t^m} \phi_r(t) \diff t=0,
\end{equation*}
we can rewrite the $\L^2$ operator in the following form 
\begin{equation}\label{eq:int_L2}
\cc{m}(t^m)\bc^m-\cc{m}(t_0)\bc^0-\sum_{r=0}^M  F(\bc^r) \int_{t^0}^{t^M} \cc{m}(t) \phi_r(t) \diff t=0,
\end{equation}
where $\cc{m}$ is the characteristic function in the interval $[t^0,t^m]$, i.\,e.,
\begin{equation}\label{eq:characteristic}
 \cc{m}(t)=\begin{cases}
         1, \qquad \text{if}  &t\in [t^0,t^m],\\
         0, \qquad \text{else}.
        \end{cases}
\end{equation}
Therefore, we can actually include the integration also for the first two terms of 
equation \eqref{eq:int_L2}, resulting in 
\begin{align}\label{eq:L_2_bases}
 &\int_{t^0}^{t^M} \cc{m}(t) \partial_t \left(\bc(t)\right) \diff t-
 \sum_{r=0}^M  F(\bc^r) \int_{t^0}^{t^M} \cc{m}(t) \phi_m(t) \diff t=0,\\
 &\int_{T^n} \psi_{m}(t) \partial_t \bc(t) \diff t- \int_{T^n} \psi_{m}(t)
 F(\bc(t)) \diff t=0, \label{eq:DeCL2asADER}
\end{align}
where $\psi_m(t)=\cc{m}(t)$ are the chosen test functions.

If we compare \eqref{eq:DeCL2asADER} with the beginning ADER formulation \eqref{eq:ADERODEL2} and \eqref{eq:L2_ADER}, we notice that they differ just in the choice of the test functions. If for the ADER approach we chose test functions to be the basis functions, in the DeC method we considered different test functions.

Inserting this in the DeC algorithm, we obtain 
\begin{equation*}
 \cc{m}(t_m)\bc^{m,(k)}=\cc{m}(t_0)\bc^0-
 \sum_{r=0}^M  F(\bc^{r,(k-1)}) \int_{t^0}^{t^M} \cc{m}(t) \phi_m(t) \diff t.
\end{equation*}
If we compare this equation with the fixed--point iteration of ADER \eqref{eq:fixpoint_iteration}, we can
observe that the discrete methods differ mainly in the mass matrix.
However, the difference in the test functions leads to different processes in the construction of the iterative matrices. 
In DeC, one cannot use integration by part in time, because the test functions are discontinuous.
This is not a problem, since the basis functions in time are $\mathcal{C}^1([t^n,t^{n+1}])$ and the derivatives can be applied directly there.  By changing the test function also in the DeC approach, we can 
 use the technique described in section \ref{sec:ader}, resulting in the 
 ADER formulation.
 \begin{conclusion}\label{eq:con}
 ADER is equivalent to DeC and \textit{vice versa}, up to few details. The major difference between the two approaches presented and used in \cite{dumbser2008unified} and \cite{abgrall2017dec} is the choice of the test functions and the resulting iteration matrix.
 \end{conclusion}

\subsection{On the Accuracy of ADER}\label{subsec:accuracy}

Knowing conclusion  \ref{eq:con}, we can finally use the DeC framework 
to prove the order condition for the ADER methods. 
Therefore, we redefine for simplicity the $\L^1$ and $\L^2$ operators, resulting in
\begin{equation}\label{eq:operators}
    \begin{cases}
        \L^1(\bbc):= \bbc -\M^{-1}\vec{r}(\bc(t^n))\\
        \L^2(\bbc):= \bbc-\M^{-1}\vec{r}(\bbc).
    \end{cases}
\end{equation} 
Following the approach from \cite{abgrall2017dec}, we
need to show that 
\begin{enumerate}[label=\textbf{C.\arabic*}]
    \item \label{item:coercivity} $\L^1$ is coercive 
    \item \label{item:lipschitz} $\L^1-\L^2$ is Lipschitz continuous with constant $\beta \Delta$ 
    \item \label{item:existence} There exists a unique solution of $\L^2$, i.e.,  $\L^2(\bbc^*) = 0$
\end{enumerate}
in order to apply the DeC Theorem.

The next proposition shows condition \ref{item:coercivity}, i.\,e., the coercivity of the operator $\L^1$. 
 
\begin{proposition}[Coercivity of $\L^1$]
Given any $\bbc, \bbd$, such that the explicit data coincides $\bc(t^n)=\bd(t^n)$, there exists a positive $C_0$ such that the 
operator $\L^1$ fulfills 
\begin{equation*}
\norm{\L_1(\bbc)-\L_1(\bbd) } \geq C_0  \norm{\bbc-\bbd}.
\end{equation*}
\end{proposition}
\begin{proof}
We remind that the beginning states coincide for all variables, i.e., $\bc(t^n)=\bd(t^n)$.
Consider any norm $\norm{ \cdot }$, from the definition \eqref{eq:L1ADER} we have
\begin{align*}
\norm{\L^1(\bbc)-\L^1(\bbd) }=\norm{\bbc -\bbd}.
\end{align*}
Therefore, we prove the statement with $C_0=1$. 
\end{proof}

With the following proposition, we prove condition \ref{item:lipschitz}, the Lipschitz continuity of the difference of operators.
\begin{proposition}[Lipschitz continuity of the $\L^1-\L^2$ operator]
Let  $\bbc,\;\bbd$ be in $\R^{M\times I}$. Then, the operator $\L^1-\L^2$ 
is Lipschitz continuous with constant $C_A:=\Delta t \beta$, i.\,e.,
\begin{equation}
 \norm{(\L^1(\bbc)-\L^2(\bbc))-\left(\L^1(\bbd)-\L^2(\bbd)\right)}\leq C_A  \norm{\bbc-\bbd} \\
\end{equation}
where $\beta$ is independent of $\Delta t$.
\end{proposition}
\begin{proof}
First, we note that the first stages 
coincide for $\bbc$ and $\bbd$, i.e., $\bc(t^n)=\bd(t^n)$.
Then, we get
\begin{align*}
&\norm{\left(\L^1(\bbc)-\L^2(\bbc)\right)-\left(\L^1(\bbd)-\L^2(\bbd)\right)} \\
= & \norm{\bbc -\M^{-1}\vec{r}(\bc(t^n)) -\bbc +\M^{-1} \vec{r}(\bbc)- \bbd +\M^{-1} \vec{r}(\bd(t^n)) +\bbd -\M^{-1} \vec{r}(\bbd) }   \\
=& \norm{ \M^{-1} \left( \vec{r}(\bbc)-\vec{r}(\bbd) \right)}\\
\leq & \norm{ \M^{-1}}  \norm{ \vec{r}(\bbc)-\vec{r}(\bbd)},
\end{align*}
where we have just used the definition of the operators and basic linear algebra properties. 
Now, we have to exploit the structure of the mass matrix \eqref{eq:MassmatrixAder} and of the right--hand side \eqref{eq:aderrhs}. 
Recalling again that $\bc(t^n)=\bd(t^n)$, the difference of the two right--hand sides can be written, for every $m = 1, ..., M$, as
\begin{align*}
&\vec{\br}(\bbc)_m-\vec{\br}(\bbd)_m \\
=& \phi_m(t^{n})\bc(t^n)
- \Delta t\sum_{z=0}^{Z} w_z \phi_m(t^q_z) F(\bphi(t^q_z)^T\bbc) -\phi_m(t^{n})\bd(t^n)
+\Delta t\sum_{z=0}^{Z} w_z \phi_m(t^q_z) F(\bphi(t^q_z)^T\bbd)\\
=&\Delta t \sum_{z=0}^{Z} w_z \phi_m(t^q_z) \left( F(\bphi(t^q_z)^T\bbd)-F(\bphi(t^q_z)^T\bbc) \right),
\end{align*}
and using the boundedness of the basis functions $\phi_m$ and of the weights $w_z$ and the Lipschitz continuity of the function $F$, we obtain that
\begin{align}
\norm{\vec{\br}(\bbc)-\vec{\br}(\bbd)} \leq \Delta t C_r \norm{\bbc - \bbd}.
\end{align}

Then, the mass matrix is constant and invertible for the common points distributions, hence, the norm of its inverse does not depend on $\Delta t$ and can be bounded by a coefficient $C_M\in \R^+$.
Overall, we can write that
\begin{align}
&\norm{\left(\L^1(\bbc)-\L^2(\bbc)\right)-\left(\L^1(\bbd)-\L^2(\bbd)\right)} \leq  \Delta t C_M C_r \norm{\bbc - \bbd} ,
\end{align}
proving the statement.
\end{proof}

\begin{theorem}[Convergence of ADER - DeC]\label{the:theorem}
The two operators as defined in \eqref{eq:operators} and used in the DeC algorithm \eqref{DeC_method} 
gives an approximate solution with order of accuracy equal to $\min(M + 1, K)$.
\end{theorem}
\begin{proof}
 Let us denote by $\bbc^*$ the solution of $\L^2(\bbc^*)=0$. We obviously have 
 \begin{equation*}
   \L^1(\bbc^*)=\L^1(\bbc^*)-\L^2(\bbc^*), 
 \end{equation*}
so that with the coercivity of the $\L^1$ operator 
\begin{align}
\label{eq:convADERproof:coerci} \norm{\bbc^{(k)} -\bbc^{*}} &\leq C_0 \norm{\L^1(\bbc^{(k)}) -\L^1(\bbc^*)}\\
 &= \norm{ \L^1(\bbc^{(k-1)})-\L^2(\bbc^{(k-1)}) - \L^1(\bbc^*) +\L^2(\bbc^*)}\\
\label{eq:convADERproof:lip} &\leq \beta \Delta t \norm{\bbc^{(k-1)}-\bbc^*},
\end{align}
where, in \eqref{eq:convADERproof:coerci} we have used the condition \ref{item:coercivity} on the coercivity of $\L^1$ and in \eqref{eq:convADERproof:lip} we have applied condition \ref{item:lipschitz} on the Lipschitz continuity of the operator $\L^1-\L^2$.
This implies that after each iteration step we obtain one order of accuracy more than in the previous 
iteration. After $K$ iterations, we finally get 
\begin{equation*}
 \norm{\bbc^{(K)} -\bbc^{*}} \leq \beta^K\Delta t^K \norm{\bc(t^n)-\bbc^*}.
\end{equation*}
Moreover, we know that $\bbc^*$ is an approximation of order $M+1$ of 
the exact solution $\bbc^{ex}$. So, overall we get
\begin{equation*}
 \norm{\bbc^{(K)} -\bbc^{ex}}\leq \norm{\bbc^{*} -\bbc^{ex}}+(\beta \Delta t)^K \norm{\bbc^{*} -\bbc^{(0)}}
 \leq C^*\left(\Delta t^{M+1}+(\beta \Delta t)^K \right)
\end{equation*}
which proves the statement of the theorem.
\end{proof}

\begin{remark}
In our study, we considered equidistant, Gauss--Lobatto and Gauss--Legendre nodes. For these types of nodes, we can guarantee 
that the mass matrix $\M$ has full rank and, hence, it is invertible. 
Questions concerning the condition numbers of these matrices and related topics will be part of future research. 
\end{remark}

\subsection{Beyond Explicitness}\label{sub_stiff}
For many problems of interest stiff source terms play an important role in the systems of equation, both in ODE and PDE cases. 
Some ADER \cite{dumbser2007FVStiff} and DeC \cite{abgrall2018asymptotic} schemes have been presented in order to handle stiff source terms in PDEs.
The key idea of these schemes is to treat implicitly the source term in the iteration process, via linearization when it is not directly invertible. 
Consider the ODE
\begin{equation}
	\frac{d}{dt} \bc +F(\bc) +S(\bc)=0 
\end{equation}
where $S:\R^I\to\R^I$ is a source term, possibly stiff, that we cannot resolve explicitly. The basic idea of \cite{abgrall2018asymptotic} is to modify the $\L^1$ operator in the DeC formulation, by considering the following discretization
\begin{equation}
	\L^1(\bbc)^m:= \bc^m-\bc^0  +\beta^m \Delta t \left( F(\bc^0) + S(\bc^m) \right),
\end{equation}
which can be further simplified in case of nonlinear terms with another first order approximation 
\begin{equation}\label{eq:L1_Stiff}
	\L^1(\bbc)^m:= \bc^m-\bc^0  +\beta^m \Delta t \left( F(\bc^0) + \partial_{\bc} S(\bc^0) \cdot \bc^m \right).
\end{equation}
The advantage of having the equations for all the sub-timesteps is twofold. First of all, the Jacobian of the stiff term does not interact with the mass matrix of the sub-timesteps, moreover, the DeC procedure still leads to a very clean form, as follows, for every $m =1,\dots, M$
\begin{align}
	\L^1(\bbc^{(k+1)})^m_i&=\L^1(\bbc^{(k)})^m_i-\L^2(\bbc^{(k)})^m_i,\\
	(\texttt{Id}_i^j+\beta^m \Delta t\partial_{\alpha_j}S_i(\bc(t^n))) \alpha_j^{m,(k+1)} &= (\texttt{Id}_i^j+\beta^m \Delta t\partial_{\alpha_j}S_i(\bc(t^n))) \alpha_j^{m,(k)} - \L^2(\bbc^{(k)})_i^m, \\
	\alpha_j^{m,(k+1)} &=  \alpha_j^{m,(k)} - (\texttt{J}^{-1})^i_j \L^2(\bbc^{(k+1)})_i^m,
\end{align}
where $\mathtt{J}_i^j:=(\texttt{Id}_i^j+\beta^m \Delta t\partial_{\alpha_j}S_i(\bc(t^n))) $ and $\mathtt{Id}$ is the identity matrix. Here we used Einstein's notation and every time an index is repeated once as superscript and once as subscript in a product, a summation is implied. In particular, we used $i,j=1,\dots, I$.

Analogously, in ADER community a similar idea was applied \cite{dumbser2007FVStiff}. It can be rewritten in the DeC formulation as follows. Let
\begin{equation}
	\L^1(\bbc):= \M \bbc - r(\bc(t^n)) +s(\bbc),
\end{equation}
where $s$ is defined as
\begin{align}
	&s(\bbc)_{m,i}= \Delta t \sum_{z=0}^Z w_z \phi_m(t^q_z) \phi_l(t^q_z) \partial_{\alpha_j} S_i(\bc(t^n)) \alpha_j^l =:  \S_{mil}^{j}\alpha_j^l  \label{eq:sourceAderJacobian}
\end{align}
when $S$ is nonlinear, where $i,j$ are the indexes of the constituents or of the equations and $m,l$ are the indexes of the basis function referring to the $M$ subtimesteps. As before, we used Einstein's notation. 
This is actually what one obtains using the Newton--Rapshon method on the $\L^2$ operator as prescribed in \cite{dumbser2007FVStiff}. Indeed, $\L^1$ comprises all the information of the Jacobian of the stiff part of the $\L^2$ operator, which is what they want to solve.

The final ADER iterative procedure sums up to
\begin{align}
	\L^1(\bbc^{(k+1)})& = \L^1(\bbc^{(k)}) -\L^2(\bbc^{(k)}), \\
	(\M_{ml}\textrm{Id}^j_i+\S_{mil}^j) \alpha^{l,(k+1)}_j \!\!\! -\vec{r}(\bc(t^n))  &=(\M_{ml}\textrm{Id}^j_i+\S_{mil}^j) \alpha^{l,(k)}_j \!\!\! -\vec{r}(\bc(t^n)) -\M_{ml}\textrm{Id}^j_i  \alpha^{l,(k)}_j \!\!\! +\vec{r}_{mi}(\bbc^{(k)})),\\
	\alpha^{l,(k+1)}_j &= \alpha^{l,(k)}_j -(\mathtt{J}^{-1})_j^{mli} \L^2(\bbc^{(k)})_{mi},
\end{align}
where $\mathtt{J}^j_{mli}:= \M_{ml}\textrm{Id}^j_i+\S_{mil}^j$.
Even if the two methods seems very similar in the building process, the implementation and the results are particularly different. In the ADER procedure the mass matrix in time is mixed with the Jacobian of the stiff part, which leads to the inversion of a matrix $(I\times M) \times (I\times M)$ at each time step, while DeC needs just the inversion of a matrix $I\times I$.
This extra effort is rewarded with extra accuracy in the solutions of the implicit ADER (IMADER) with respect to the implicit DeC (IMDeC), see Section \ref{sec:StiffSimulations}. 
We can heuristically explain this behavior with the fact that in the explicit processes, most of the terms in the $\L^1$ operators cancel out, while in the implicit versions they play a more important role. 
The fact that IMDeC considers these terms in a very simplified first order manner degrade the accuracy of the solutions, while IMADER performs a more accurate integration in time for the $\L^1$ operator, obtaining a more coupled mass matrix. 
Heuristically speaking, IMADER is closer to a pure implicit schemes than IMDeC. 
In future research, we will further analyze this behavior. Especially, a comparison 
between the different interpretations of IMADER and IMDeC and their use as time integration schemes for PDEs with stiff source terms is desirable. This yields us straight to the next subsection.   
\subsection{Beyond ODEs} \label{sec:beyondODE}

The ADER method has been introduced for hyperbolic problems and the DeC method is also used in the hyperbolic community. In this section, we shortly repeat the DeC and ADER variations and their applications to PDEs. Here, we focus on three different versions: 
\begin{enumerate}
\item DeC - residual distribution  schemes
\item ADER space-time discontinuous Galerkin 
\item ADER(DeC) - spectral difference schemes
\end{enumerate}

Points one and two are -at least from our knowledge- the most used applications of DeC and the modern ADER in the context of time integration scheme whereas ADER(DeC) - spectral difference (SD)  can be seen as a variation of the classical ADER(DeC) approach, successfully applied to solve the magnetic induction equation in \cite{veiga2020arbitrary}. Later in our numerical section, we will compare the performances of ADER-SD and DeC-SD focusing on two simple hyperbolic test cases, e.g. linear advection and Burgers' equation, to verify and support our theoretical findings also for the PDE framework. 

Stability will only be considered in the ODE case focusing on A-stability. As part of future research, a stability analysis (von Neumann and/or entropy stability) of the fully discrete schemes would be desirable but this is beyond the scope of this work.

\subsubsection{DeC as Time Integration Scheme}

The main advantage of DeC as a time integration scheme was pointed out by Abgrall in \cite{abgrall2017dec}. 
He showed that by using the DeC approach, one can avoid the inversion of the mass matrix in continuous 
finite element approximations.  
This is favorable since the mass matrices are usually sparse and hard to invert which is different when 
considering DG schemes where the mass matrix has some block diagonal structure and is easy to invert. 
Abgrall works with residual distribution (RD) schemes. RD puts several high order schemes like continuous and discontinuous Galerkin and Flux Reconstruction (FR) into a common framework \cite{abgrall2018connection,abgrall2020error}.
Further, RD works directly on the degrees of freedom and, through this abstract approach, the use of unstructured grids (e.g. triangles or general polygons) is straightforward. Simultaneously, these advantages do not come without 
some  drawback. The main disadvantage of the RD approach is that applying the method of lines would destroy the high order accuracy of the scheme \cite{abgrall2006residual}. Due to this fact, RD is usually introduced for steady state equations
\begin{equation}\label{eq:steady_state}
\div f(x)=0. 
\end{equation}
To explain shortly the update for a hyperbolic conservation law \eqref{eq:scalarPDE},
we denote by  $\mathcal{R}_p^K$ the space residual for a generic element $K$ (as usual in FE, the domain was split into subdomains) and degree of freedom $p$ which solves the steady state equation \eqref{eq:steady_state} in $p$.
To achieve finally a space-time discretization, the classical DeC from Section \ref{sec:originalDeC} is applied
and we get 
\begin{equation}\label{oneline}
U_p^{m,(k+1)} = U_p^{m,(k)} -
|C_p|^{-1} 
\sum_{K|p \in K}\bigg(\int_K \Phi_p(x) \left(U^{m,(k)}(x) - U^{n,0}(x)\right) \dd x + \Delta t
\sum_{r=0}^M \theta_{r}^m \mathcal{R}_p^K(U^{r,(k)}) \bigg),
\end{equation}
where 
$
|C_p| := \int_K \Phi_p (x)\diff x 
$
has to be strictly positive and $\Phi$ are the spatial basis functions. This term comes from the further first order simplification of the mass lumping operated in the $\L^1$ operator, c.f.  \cite{abgrall2017dec, abgrall2019high,abgrall2018asymptotic} for details. 

As one can recognize the scheme is explicit and as demonstrated in \cite{abgrall2017dec, abgrall2019high}
high order in space and time. An IMEX-DeC-RD variant can be found in 
\cite{abgrall2018asymptotic} and a comparison with IMEX-ADER-RD will be considered in the future. Finally, we want to mention that
DeC has already been used as a time integration scheme in combination with DG in \cite{liu2008strong} where 
the method of lines have been used to split the space-time discretization.

\def\STC{T^n\times V_i}
\def\SC{V_i}
\def\TC{T^n}
\subsubsection{ADER Space--time DG}
Up to now, we have presented the modern and historical ADER approach from our point of view as an ODE solver. 
However, to clarify again that our description of the modern ADER used in \ref{sec:ader} is indeed equivalent with the one used \textit{inter alia} in \cite{dumbser2007FVStiff,zanotti2015space}, we focus again on it. We hope that the understanding for ADER will be simpler with this different perspective.\\
We follow now \cite{zanotti2015space} where ADER is presented as a space--time DG approach for hyperbolic problems.
This is not the unique version of ADER for such problems. In particular, there exist implicit versions, different spatial discretizations, even different ways of integrating the spatial discretization within the time discretization that may change it in some aspects.
Nevertheless, we study the DG in space and time to explain one extension of ADER to hyperbolic PDEs.
It uses time and space test functions, which are the tensor products of basis functions in time and basis functions in space. 
We can write their formulation in our setting.
Hence, we define the new basis functions 
\begin{equation}\label{eq:basisSpaceTime}
\theta_{pq}(x,t):=\Phi_p(x)\phi_q(t),
\end{equation}
 where $\Phi$ are the basis functions in space and $\phi$ are the basis functions in time. With the Einstein summation notation, the reconstruction variable is
\begin{equation}\label{eq:reconstructionTimeSpace}
u(x,t)=u^{pq}\theta_{pq}(x,t)=\sum_{p}\sum_q u^{pq} \Phi_p(x) \phi_q(t).
\end{equation}

Let us consider a hyperbolic problem given as
\begin{equation}
\label{eq:pde}
\partial_t u(x,t) + \nabla \cdot F(u(x,t)) = 0, \qquad  x\in \Omega\subset \R^d,\; t>0
\end{equation}
with appropriate initial and boundary conditions.
 We consider the weak solution of \eqref{eq:pde} in space--time. Let us define a space--time cell, $\STC$, given by the tensor product of a timestep $T^n$ and a volume cell $V_i\subset \R^d$. We multiply \eqref{eq:pde} by test function $\theta_{rs}(x,t)$ and integrate over the defined control volume:
\begin{equation}
\int_{\STC} \theta_{rs}(x,t)\partial_t \theta_{pq}(x,t) u^{pq} \diff x \diff t+ \int_{\STC} \theta_{rs}(x,t) \nabla \cdot F(\theta_{pq}(x,t) u^{pq})  \diff x \diff t=0.
\end{equation}
We also introduce the following definitions
\begin{equation}
u^q(x) = \Phi_p(x) u^{pq}, \qquad u^{p}(t):=\phi_q(t) u^{pq}.
\end{equation}
Now, we apply integration by parts in time and obtain
\begin{equation}
\begin{split}
0=&\int_{\SC} \Phi_r(x) \Phi_p(x) \diff x \left( \phi_s(t^{n+1}) \phi_q(t^{n+1}) u^{pq} - \phi_s(t^{n}) u^{p}(t^n)  -\int_{\TC} \partial_t \phi_s(t) \phi_q(t) \diff t  \; u^{pq} \right)+\\
&\int_{\TC} \phi_s(t)\phi_q(t) \diff t  \int_{\SC} \Phi_r(x)   \nabla \cdot  F(u^q(x)) \diff x .  
\end{split}
\end{equation}
Here, the spatial integral is not treated with integration by parts, in order to obtain a fully local method.

We define the spatial mass matrix and the time mass matrix as
\begin{equation}
\M^x_{rp}:=\int_{\SC} \Phi_r(x) \Phi_p(x) \diff x, \qquad \M^t_{sq}:=\int_{\TC} \phi_s(t)\phi_q(t) \diff t.
\end{equation}
Splitting as done in \eqref{eq:System} into a linear left--hand side and a non--linear right--hand side, we obtain
\begin{equation}
\begin{split}
&\M^x_{rp} \left( \phi_s(t^{n+1}) \phi_q(t^{n+1})   -\int_{\TC} \partial_t \phi_s(t) \phi_q(t) \diff t  \right)u^{pq}= \M^x_{rp} \phi_s(t^{n}) u^{p}(t^n) - \M^t_{sq} \int_{\SC} \Phi_r(x)   \nabla \cdot  F(u^q(x)) \diff x.
\end{split}
\end{equation}

Using the Picard iteration process we obtain, again
\begin{equation}\label{eq:ADER_DG}
\vec{\vec{\M}}_{rspq} u^{pq,(k+1)} = \vec{\vec{r}}(\vec{\vec{\mathbf{u}}}^{(k)})_{rs},
\end{equation}
where 
\begin{align}
&\vec{\vec{\M}}_{rspq}:=\M^x_{rp} \left( \phi_s(t^{n+1}) \phi_q(t^{n+1})   -\int_{\TC} \partial_t \phi_s(t) \phi_q(t) \diff t  \right), \label{eq:ADERDG_mass_matrix}\\
&\vec{\vec{r}}(\vec{\vec{\mathbf{u}}}^{(k)})_{rs}\!\! :=
\M^x_{rp} \phi_s(t^{n}) u^{p}(t^n)\! - \M^t_{sq}\int_{ \SC} \Phi_r(x) \nabla \cdot  F(u^q(x)) \diff x . \label{eq:ADERDG_rhs}
\end{align}

A final step must be done in order to guarantee the communication between neighboring cells, approximating the solutions of generalized Riemann problems as follows
\begin{equation}
\M_{rp} \left( u(t^{n+1})_r-u(t^n)_r \right) + \int_{ \TC} \int_{\partial \SC} \Phi_r(x) \mathcal{G}(u^{(K),-},u^{(K),+}) \cdot \boldsymbol{\mathrm{n}} \, \diff S\, \diff t - \int_{ \TC} \int_{ \SC} \nabla \Phi_r \cdot F(u^{(K)}) \, \diff x\, \diff t =0,
\end{equation}
where $\mathcal{G}$ is a classical numerical two-point flux, $u^+$ is the value of $u$ inside $\SC$ and $u^-$ is the value of $u$ in the neighboring cell, $u^{(K)}(x)$ is the reconstruction of $u$  at $K$th iteration of the Picard--Lindel\"of process and the normal vector $\boldsymbol{\mathrm{n}}$ is pointing outwards. 

\begin{remark}
	The approximation of the generalized Riemann problem used only in the final step allows to obtain a local scheme for the whole iterative process and let the cell communicate at the last stage, allowing an easy parallelization.
	Nevertheless, also using the weak formulation of DG with a numerical flux that allows communication in every stage of the process is possible. This would even save the last update step, though losing the locality of the algorithm. \\
	The advantage of this approach solving the generalized  Riemann problem only in the final step 
	have been analyzed in the works  \cite{balsara2013efficient,dumbser2018efficient} where  comparisons to classical RKDG schemes in terms of performance and runtime can be found.  
	Here, it was recognized that ADER-DG can keep up or be even more effective than traditional RKDG schemes where at every stage, a communication between each element (solving a generalized Riemann problem) have to be done. Therefore, at each stage every spatial degree of freedom has to be touched. 
\end{remark}

\begin{remark}
We presented here in detail the ADER--DG method.  
Considering \eqref{eq:ADERDG_mass_matrix}, one notice that the full space--time matrix, actually depends on several different mass matrices, i.e., on $\M^x_{rp}$ and on $\M^t_{sq}$. 
Changes in the space discretization affect only the space discretization matrix $\M^x_{rp}$, one can, for example, consider different schemes like Flux Reconstruction which are working with the differential formulation of the PDE.
Nevertheless, the time--integration structure of the matrix \eqref{eq:ADERDG_mass_matrix} given by the ADER approach will stay the same. That is why we can consider  ADER to be a time integration method.
\end{remark}

\subsubsection{ADER/DeC--Spectral Difference Schemes}\label{sec:SD}

In this section we briefly describe the spectral difference scheme (as in \cite{Liu2006}). Let us consider again the hyperbolic problem given in \eqref{eq:pde}. We focus on the description of the solution in one element (for a fixed time),

\[ u(x) = \sum_{r=0}^N u(x^r)\Phi_r(x), \]
 which is given by Lagrange interpolation polynomials $\{\Phi_{r}(x)\}_{r=0}^N$, built on a set of points $\mathcal{S}^s=\{x^r\}_{r=0}^N$, called the solution points, with $N$ being the polynomial degree of the interpolation Lagrange polynomials.

The flux is approximated by another set of Lagrange interpolation polynomials $\{\xi_{p}(x)\}_{p=0}^{N+1}$ built on a second set of nodes $\mathcal{S}^f=\{y^p\}_{p=0}^{N+1}$ (flux points). Because the first and last flux points coincide with the boundary of the elements ($y^0$ and $y^{N+1}$), a numerical flux based on a Riemann solver must be used to enforce the continuity of the flux between elements. 

Let $\hat{f}(\cdot)$ denote this single-valued numerical flux, common to the element and its direct neighbor. The  approximation for the flux is given by: 
\begin{equation}
\label{eq:numerical_flux}
f(x) = \hat{f}(u(y^0)) \xi_0(x) + \sum_{p=1}^{N} f(u(y^p)) \xi_p(x) +  \hat{f}(u(y^{N+1})) \xi_{N+1}(x).
\end{equation}

The final update of the solution is obtained using the exact derivative of the flux evaluated at the solution points, so that the semi-discrete scheme reads:
\begin{align*} 
\frac{{\rm d}}{{\rm d}t} u(x^r) = - \hat{f}(u(y^0))  \xi^{\prime}_0(x^r) -\sum_{p=1}^{N} f(u(y^p)) \xi^{\prime}_p(x^r) - \hat{f}(u(y^{N+1})) \xi^{\prime}_{N+1}(x^r),
\end{align*}
where the primes stand for the derivative of the Lagrange polynomials.

To evolve in time, one can apply the ADER or DeC scheme. We first describe the ADER scheme, based on a Galerkin projection in time.  We multiply the previous conservation law by an arbitrary test function $\psi(t)$, integrating in time over the interval $[t^n,t^{n+1}]$. We update the solution $u$ in the solution points $u(x^i,t)$, and we define $\bu=(u^0,\dots,u^N)$ the vector of all these values, where $u^r=u(x^r)$. The update reads
\[\int^{t^{n+1}}_{t^{n}}  \psi(t)\partial_t \bu {\rm d}t +\int^{t^{n+1}}_{t^{n}}   \psi(t)\partial_x f(\bu)  {\rm d}t = 0,\]
where $\partial_x f(\bu)$ is given by \eqref{eq:numerical_flux} evaluated in the solution points.
Integrating by parts yields:
\begin{equation}
\label{eq:ADER_ibp}
\psi(t^{n+1}) \bu(t^{n+1}) - \psi(t^{n}) \bu(t^{n})  
-\int^{t^{n+1}}_{t^{n}}  \partial_t \psi(t) \bu(t) {\rm d}t
+\int^{t^{n+1}}_{t^{n}}  \psi(t) \partial_x f(\bu(t)) {\rm d}t = 0.
\end{equation}

We now represent our solution using Lagrange polynomials {\it in time }$\phi_i(t)$ defined on $M+1$ Legendre quadrature points 
$\lbrace t_s \rbrace_{s=0}^M \in [t^n,t^{n+1}]$, which together with the quadrature weights $\lbrace w_s \rbrace_{s=0}^M$
can be used to perform integrals at the correct order in time. We are aiming at a solution with the same order of accuracy in time
and in space, so $M$ is taken in order to match the spatial accuracy.
We can write:
\[ \bu(t) = \sum_{s=0}^M \bu^{s} \phi_s(t),\]
and replace the integrals in \eqref{eq:ADER_ibp} by the respective quadratures. We now replace the arbitrary test function $\psi(t)$
by the set of Lagrange polynomials $\{\phi_q(t)\}_{q=0}^M$ and obtain:
\begin{equation}\label{eq:System2}
\phi_q(\tp)\left(\sum_{s=0}^{M} \bu^{s} \phi_s(\tp)\right) - \phi_q(\tn)\bu(\tn)
- \Delta t \sum_{s=0}^{M} w_s \phi^\prime_q(t_s) \bu^{s} 
+ \Delta t \sum_{s=0}^{M} w_s \phi_q(t_s) \partial_x f(\bu^{s}) 
=0 .
\end{equation}

Then, we can write the mass matrix $\M \in \R^{(M+1)\times(M+1)}$ 
and a right-hand side vector $r$ as:
\begin{equation}\label{eq:MassmatrixAder2}
\M_{qs} = \phi_q(\tp)\phi_s(\tp)- \Delta t w_s \phi^\prime_q(t_s)~~~{\rm and}~~~r_q = \phi_q(\tn)\bu(\tn) - \Delta t \sum_{s=0}^{M} w_s \phi_q(t_s) \partial_x f(\bu^s).
\end{equation}
The previous implicit nonlinear equation with unknown $\lbrace \bu^s \rbrace_{s=0}^M$, if we define $\bbu:=(\bu^0, \dots, \bu^M)$, is now written as:
\begin{equation}\label{fix:point2}
 \M_{qs} \bu^{s} = r_q(\bbu),
\end{equation}
which can be solved with a fixed-point iteration method. 

The final predicted states $\lbrace \bu^s \rbrace_{s=0}^M$ evaluated at our quadrature points are used to update the final solution as:
\begin{equation}
\bu(\tp) = \bu(\tn) - \Delta t \sum_{s=0}^M w_s \partial_x f(\bu^s).
\end{equation}
Note that in this version of the ADER scheme, we need to estimate the derivative of the flux for each time slice according to the SD method, including the Riemann solvers at element boundaries. This differs from the ADER flavor presented above, which remains local (without boundary evaluation) until the final update. 

To perform the time evolution using the DeC scheme, we write the update iterations as:

\begin{equation}
\bu^{s,{(k+1)}} = \bu^{s,(k)} - \Delta t \sum_{q=0}^M \theta_{q}^s  \partial_x f( \bu^{q,(k)} ), \quad k = 0, ..., K-1.
\end{equation}

And the final solution at $t^{n+1}$ is given as:

\[ \bu(\tp) = \bu^{M,(K)} .\]

\section{Numerics}\label{sec:Num}
\label{sec:results}

In this section, we verify our analysis by numerical simulations and also compute the stability region for ADER and DeC 
in this context. Here, we focus on A-stability for ODEs and  this  is up to our knowledge the first time that these stability conditions are investigated on ADER schemes.  Next, we compare the performance of DeC and ADER in the ODE case for linear 
and nonlinear scalar equations as well as ODE systems. 
First, we focus on the explicit ADER and DeC methods and then we consider the implicit versions which are described in Section \ref{sub_stiff}.  
To finish this section, we apply both approaches (explicit DeC and ADER) to simple PDE cases for the sake of completeness, where the space discretization 
is done via a spectral difference method. 

\subsection{Stability Conditions}
We study and compare the A-stability property of ADER and DeC.
We have to mention that the stability for different kind of DeC methods is already studied in the first work on DeC by Dutt et al. 
\cite{dutt2000dec} but, here, we investigate for the first time the simplified version introduced by Abgrall \cite{abgrall2017dec}. 
Even if the differences are small, we present the results for the sake of completeness. \\
Consider  the test problem
\begin{equation}
	\begin{split}
		\label{eq:scalar-linear}
		y'(t) &= \lambda y(t)\\
		y(0) &= 1
	\end{split}
\end{equation}
where $\lambda$ is a complex number. Let us define with $\phi$ the stability function for any method, see \cite{butcher08numODE}.
In figure \ref{fig:stab}, the left picture shows the stability regions $\mathcal{S}:=\lbrace z \in \mathbb{C} : \, \phi(z) < 1 \rbrace$ for ADER and DeC methods using Gauss--Lobatto 
nodes and different orders. The stability regions mapped each other and no differences can be seen.
In the right picture of figure \ref{fig:stab}, we show the stability regions for both ADER and DeC methods, using different collocation points for the subtimesteps, namely, equispaced, 
Gauss--Legendre and Gauss--Lobatto nodes. We verify that the choice of the collocation points for the substeps does not interfere with the stability region. 
Furthermore, both methods seem identical for different orders (2 to 5). Further studies revealed that the different subtimestep locations (Gauss-Legendre, 
Gauss-Lobatto and equispaced points) yielded no difference in the stability region for explicit DeC and ADER.

\begin{figure}[h]
    \centering
    \includegraphics[width=0.45\textwidth]{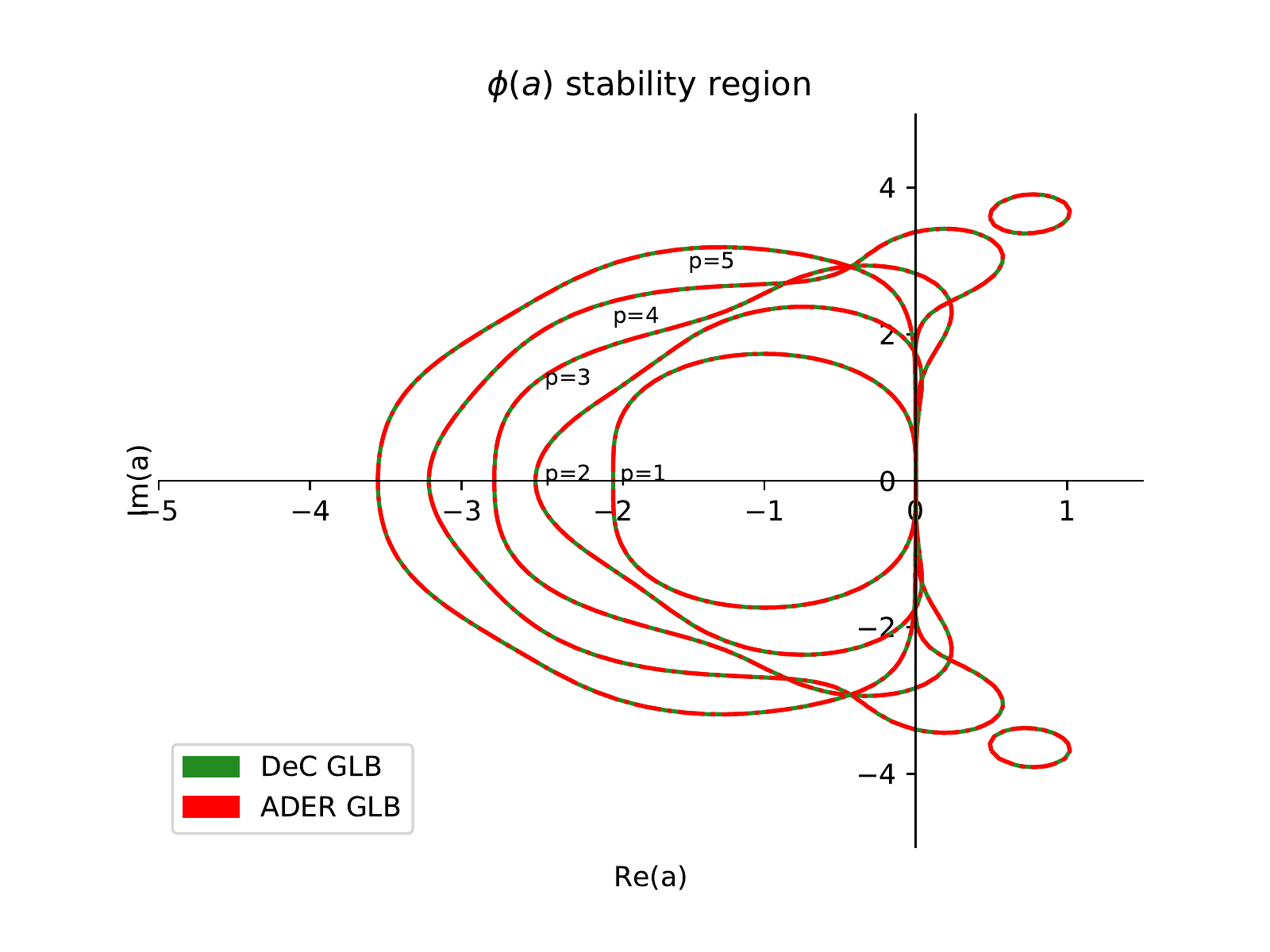}
    \includegraphics[width=0.45\textwidth]{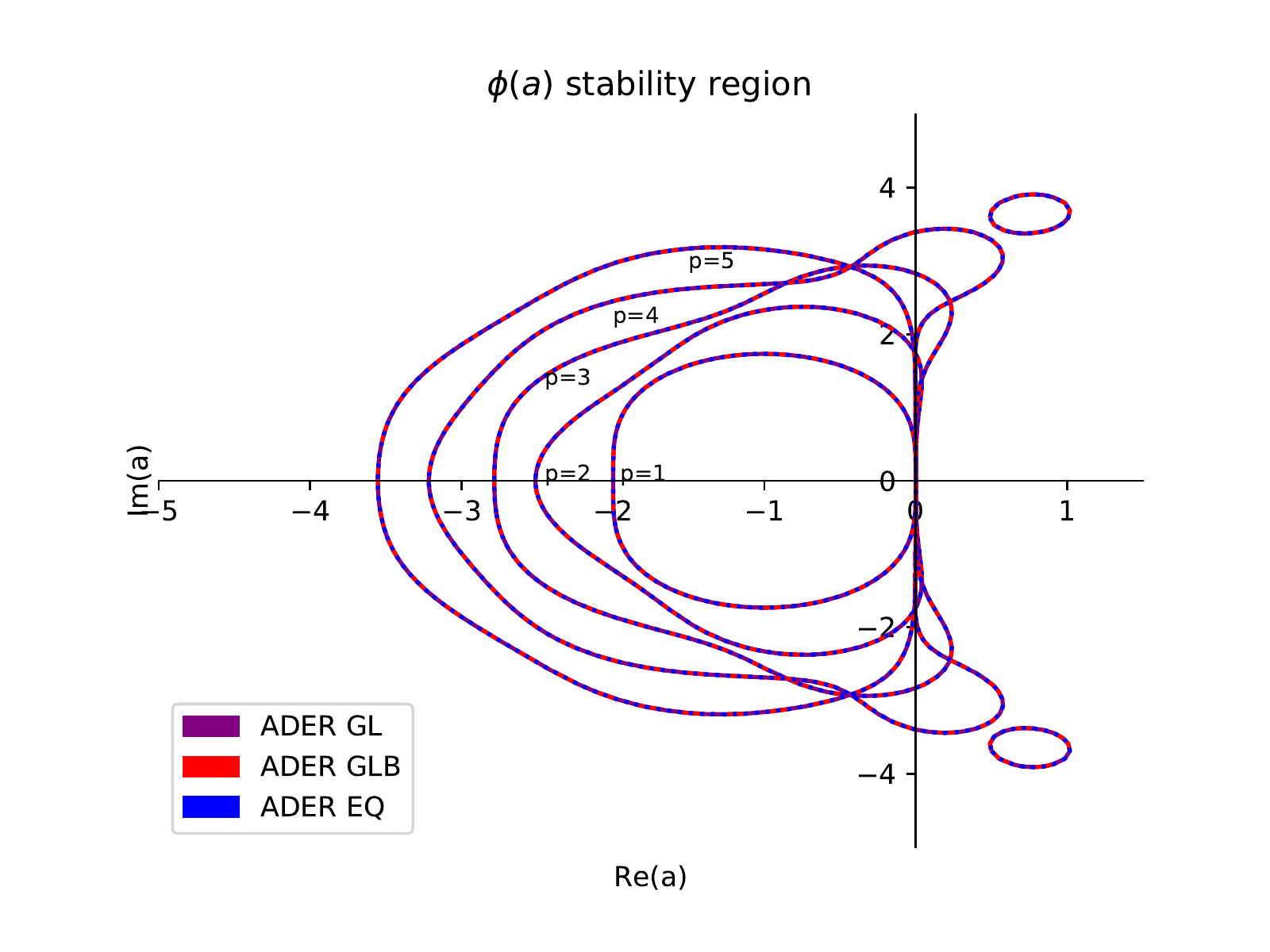}
    \caption{Stability region}
    \label{fig:stab}
\end{figure}

Furthermore, one can notice in figure \ref{fig:stab} that, as the order of accuracy increases, the stability region grows, as it is common for many other ODE solvers.

The pictures are obtained with numerical simulations of \eqref{eq:scalar-linear} for one time step with complex values of $\lambda$ on an equispaced grid and $\Delta t=1$. 
A purely analytical study of the A-stability for these methods is not feasible since it leads to solving polynomial inequalities of degree $p(p+1)$.

As mentioned above, here we analyzed the stability properties of ADER and DeC as some classical time integration scheme and focused on A-stability. 

Since both methods are usually applied for hyperbolic problems, one could be interest in finding CFL bounds from this type of analysis. Nevertheless, this is not directly possible, since they will strongly depend on the spatial discretization. An extension of the stability analysis to von Neumann stability or energy (entropy) stability, as the ones in \cite{abgrall2019analysis, glaubitz2018application}, would be desirable and it is planned for future research. 
For instance, a theoretical von Neumann stability analysis for DeC with RD spatial discretization using different polynomials (Lagrangian and Bernstein) with different stabilization techniques can be found in the PhD thesis of one of the authors \cite[Section 4.2]{torlo2020hyperbolic}, whereas linear stability and nonlinear
stability for different ADER--DG variations have been investigated numerically in \cite{dematte2019ader}.

\subsection{Convergence Error}
In this section we report on the convergence rates of DeC and ADER, for different nodal placement of the subtime steps and for different problems.

We compute the absolute discrete $L^2$ error taken over all the constituents (e.g. if considering systems of ODEs) and all the timesteps $\{t^n\}_{n=0}^N$.

\[ E =  \left( \frac{1}{N} \sum_{n=1}^N \frac{1}{I} \sum_{i=1}^I \left( y_i(t^n)-y_i^n \right)^2\right)^{\frac{1}{2}}. \]

Whenever there is no analytical solution readily available, we use  ODE integrators written in 
Julia \cite{rackauckas2017differentialequations} to compute a high accuracy numerical solution. 
These integrators use a couple of routines to select a proper numerical method for the specific type
of problem. The error tolerance is set to $10^{-15}$. 

\subsubsection{Scalar Cases}
We start by considering the same simple linear, scalar test case as in \eqref{eq:scalar-linear}, with the initial condition $y(0) = 1$, $\lambda=10$ for $t\in[0,0.1]$.
figure \ref{fig:scalar-linear-conv} shows the error convergence for the DeC and ADER methods, using Gauss-Lobatto nodal
placement and equidistant and Gauss-Legendre nodal placements, respectively. We note that at 
lower orders and when the error is far enough from machine precision, all methods seem to behave the same. 
However, at the highest reported orders, we note that ADER EQ (denoting the equispaced nodes)
does not converge with the right order. We suspected that the cause of this is the higher condition number of the mass matrix produced with
equispaced nodes. It is well known that Gauss--Legendre and Gauss--Lobatto nodes have better stability properties. DeC is not 
affected by this particular problem because it does not require a mass matrix. 
Another problem that may arise are the negative weights related to the equispaced nodes. They are present from ninth order on and also the classical Runge phenomena can appear. This is also true, of course, for ADER EQ.
Furthermore, ADER GL (denoting Gauss--Legendre nodes) 
also shows a strange behavior at high orders and high resolutions. We believe this is because of the required
extrapolatory step to compute the solution at $t^{n+1}$, only necessary when using nodes that do not include the boundaries of the interval.
The difference between ADER GLB and DeC GLB seem to be negligible.

Next, we consider the  nonlinear scalar problem

\begin{equation}
\label{eq:scalar-nonlinear}
\begin{split}
y'(t) &= -k |y(t)| y(t) \\
y(0) &= 1,
\end{split}
\end{equation}
with $k=10$ and $t\in[0,0.1]$.

Figure \ref{fig:scalar-nonlinear-conv} shows the error convergence for the DeC and ADER methods, using Gauss--Lobatto nodal placement, equidistant and Gauss--Legendre nodal placements, respectively.
The behavior of the different schemes is similar to the one reported in the scalar linear case.

\begin{figure}
\begin{minipage}[c]{0.45\linewidth}
\includegraphics[width=\linewidth]{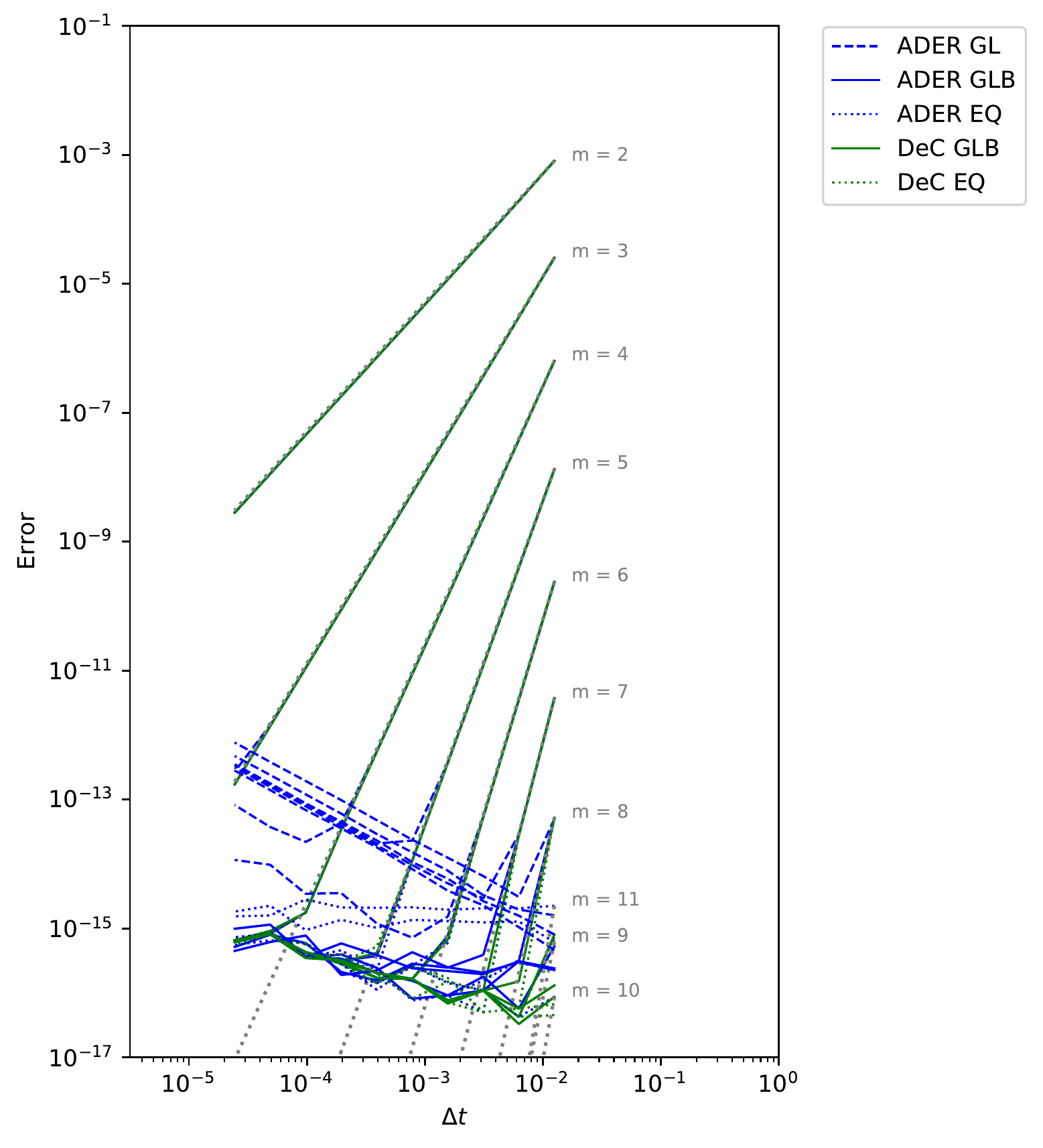}
\caption{Convergence curves for ADER and DeC, varying the approximation order and collocation of nodes for the subtimesteps for a scalar linear ODE \eqref{eq:scalar-linear}.}
\label{fig:scalar-linear-conv}
\end{minipage}
\hfill
\begin{minipage}[c]{0.45\linewidth}
\includegraphics[width=\linewidth]{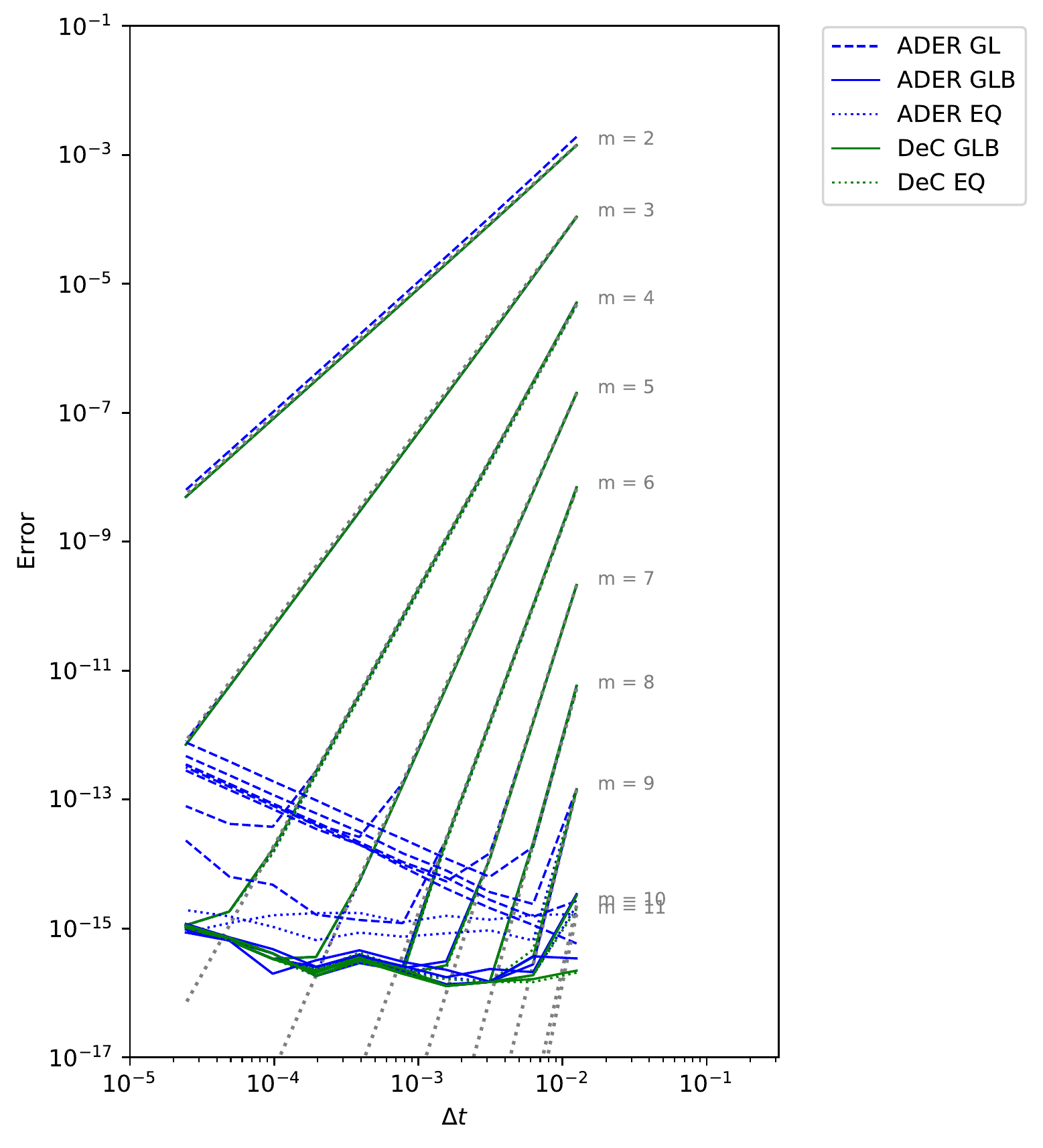}
\caption{Convergence curves for ADER and DeC, varying the approximation order and collocation of nodes for the subtimesteps for a scalar nonlinear ODE \eqref{eq:scalar-nonlinear}.}
\label{fig:scalar-nonlinear-conv}
\end{minipage}%
\end{figure}

\subsubsection{Systems}
We consider a simple biological model that models the transfer of biomass, given by a set of linear ODEs
\begin{equation}
\label{eq:system-linear}
\begin{split}
y_1'(t) &= -y_1(t) + 3y_2(t)\\
y_2'(t) &= -3y_2(t) + 5y_3(t)\\
y_3'(t) &= -5y_3(t),
\end{split}
\end{equation}

with the initial conditions
\[ (y_1(0), y_2(0), y_3(0)) = (0, 0, 10),\]
and we let $t\in[0,1]$.

The analytical solution is given by
\begin{align*}
y_1(t) &= \frac{15}{8} y_3(0) \left( e^{-5t} - 2e^{-3t} + e^{-t} \right) \\
y_2(t) &= \frac{5}{2} y_3(0) \left(-e^{-5t} + e^{-3t} \right) \\
y_3(t) &=  y_3(0) e^{-5t}.
\end{align*}

In figure \ref{fig:system-linear-conv}, we show the error convergence for the DeC and ADER methods using different nodal placements. Again, everything is as expected and we obtain the desired convergence rates. Only minor differences can be observed between the different methods.

\subsubsection*{Lotka--Volterra Equations}
Next, we consider the Lotka--Volterra equations \cite{nicolas2011} that describe the dynamics of a two--species system in which one is a predator and the other its prey. The following nonlinear equations describe the dynamics of the prey ($y_1$) and the predator ($y_2$):
\begin{equation}
\label{eq:system-nonlinear}
\begin{split}
y_1'(t) &= \alpha y_1(t) - \beta y_1(t) y_2(t) \\
y_2'(t) &= -\gamma y_2(t) + \delta y_1(t) y_2(t)
\end{split}
\end{equation}
where $\alpha$ is the growth rate of the prey, $\beta$ the predation rate, $\delta$ the predator food conversion efficiency and $\gamma$ the predator mortality.

We use the following initial conditions and parameters
\begin{equation}
\label{eq:system-nonlinear-ic}
\begin{split}
\left( y_1 (0), y_2(0) \right) &=  (1,2) \\
\left( \alpha, \beta, \delta, \gamma \right) &= (1,0.2,0.5,0.2),
\end{split}
\end{equation}
and we let $t\in[0,5]$.

Again, figure \ref{fig:lotka-conv} shows the error convergence for the DeC and ADER methods, using different nodal placements. In figures \ref{fig:lodka-sol-ader} and \ref{fig:lodka-sol-dec}, we show the solution to the initial conditions \eqref{eq:system-nonlinear-ic} and for $t \in [0,100]$, comparing the performance of ADER and DeC (with Gauss--Lobatto nodes) at different orders and resolutions. As expected the the two schemes do not exhibit large differences between the schemes, and one can notice that higher orders  produce more accurate solutions.

\begin{figure}
\begin{minipage}[c]{0.45\linewidth}
\includegraphics[width=\linewidth]{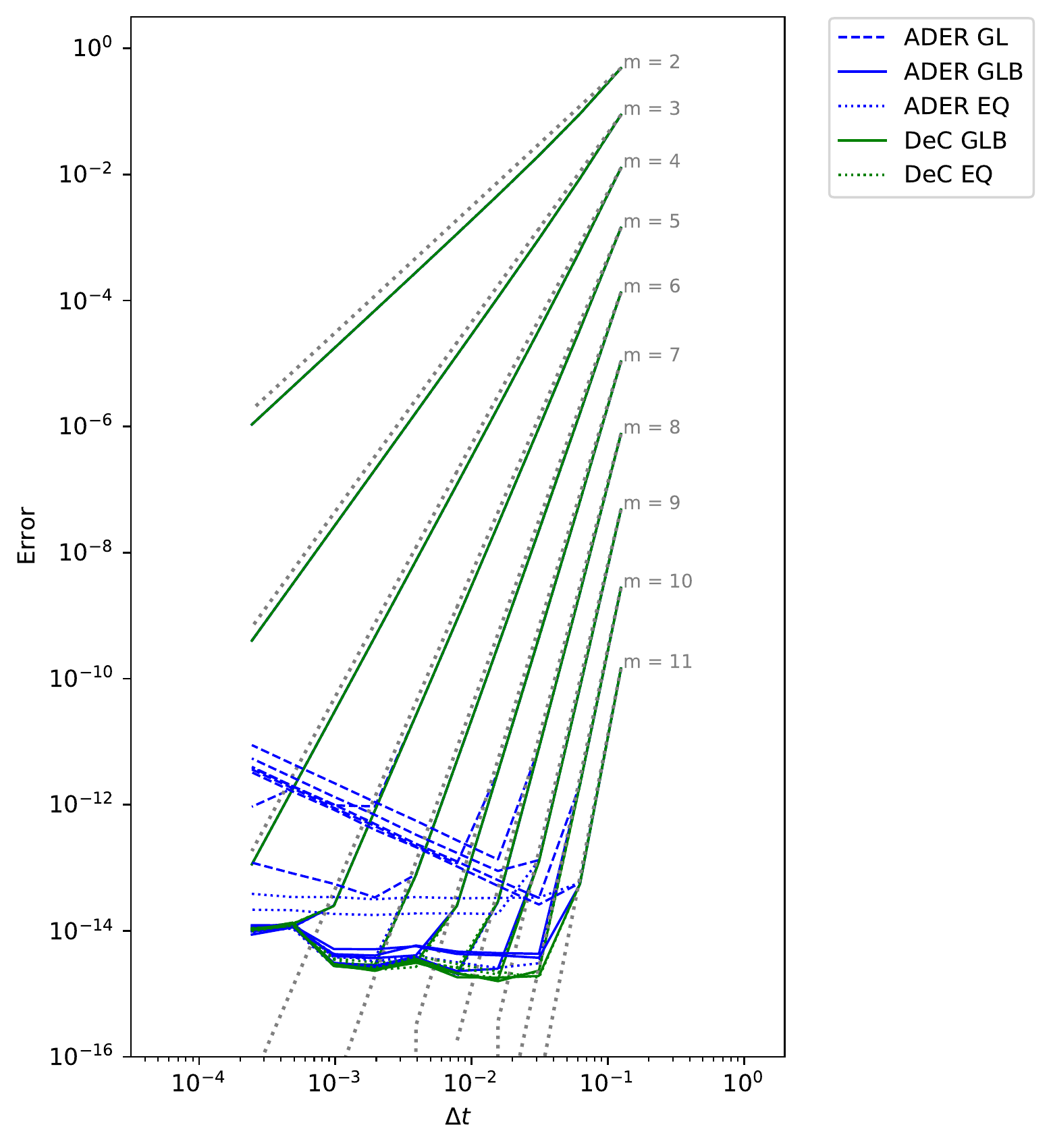}
\caption{Convergence curves for ADER and DeC, varying the approximation order and collocation of nodes for the subtimesteps for the system of linear ODEs \eqref{eq:system-linear}.}
\label{fig:system-linear-conv}
\end{minipage}
\hfill
\begin{minipage}[c]{0.45\linewidth}
\includegraphics[width=\linewidth]{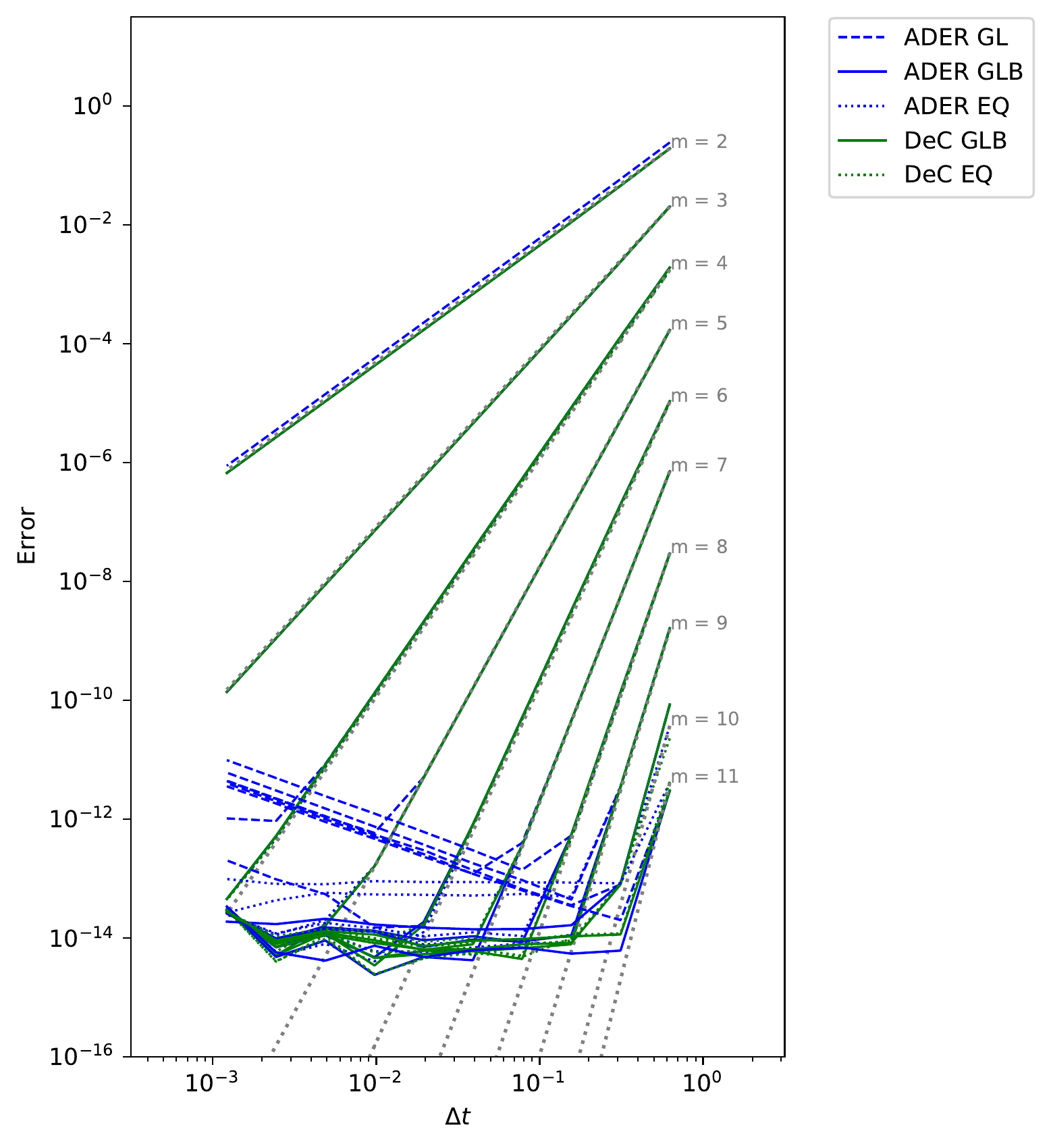}
\caption{Convergence curves for ADER and DeC, varying the approximation order and collocation of nodes for the subtimesteps for the system of nonlinear ODEs \eqref{eq:system-nonlinear}.}
\label{fig:lotka-conv}
\end{minipage}%
\end{figure}

\begin{figure}
	\begin{center}
\includegraphics[width=0.65\linewidth]{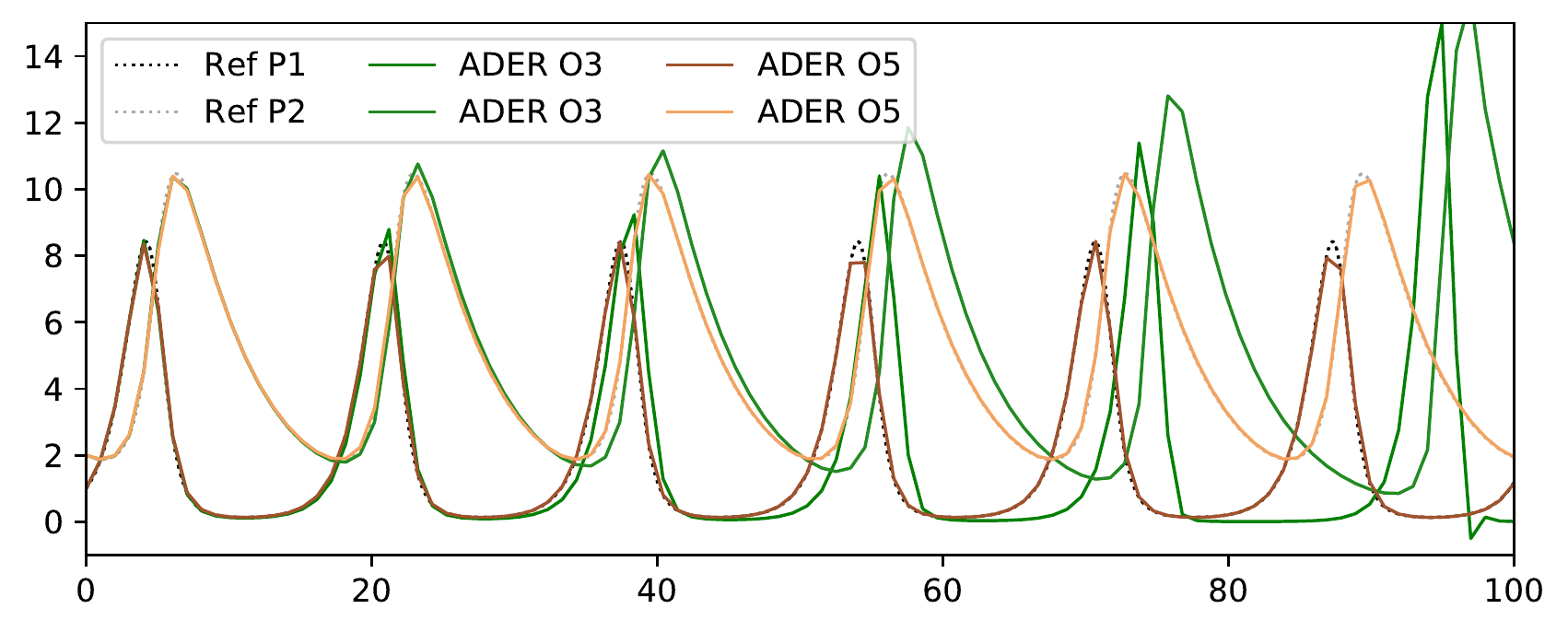}\\
\includegraphics[width=0.65\linewidth]{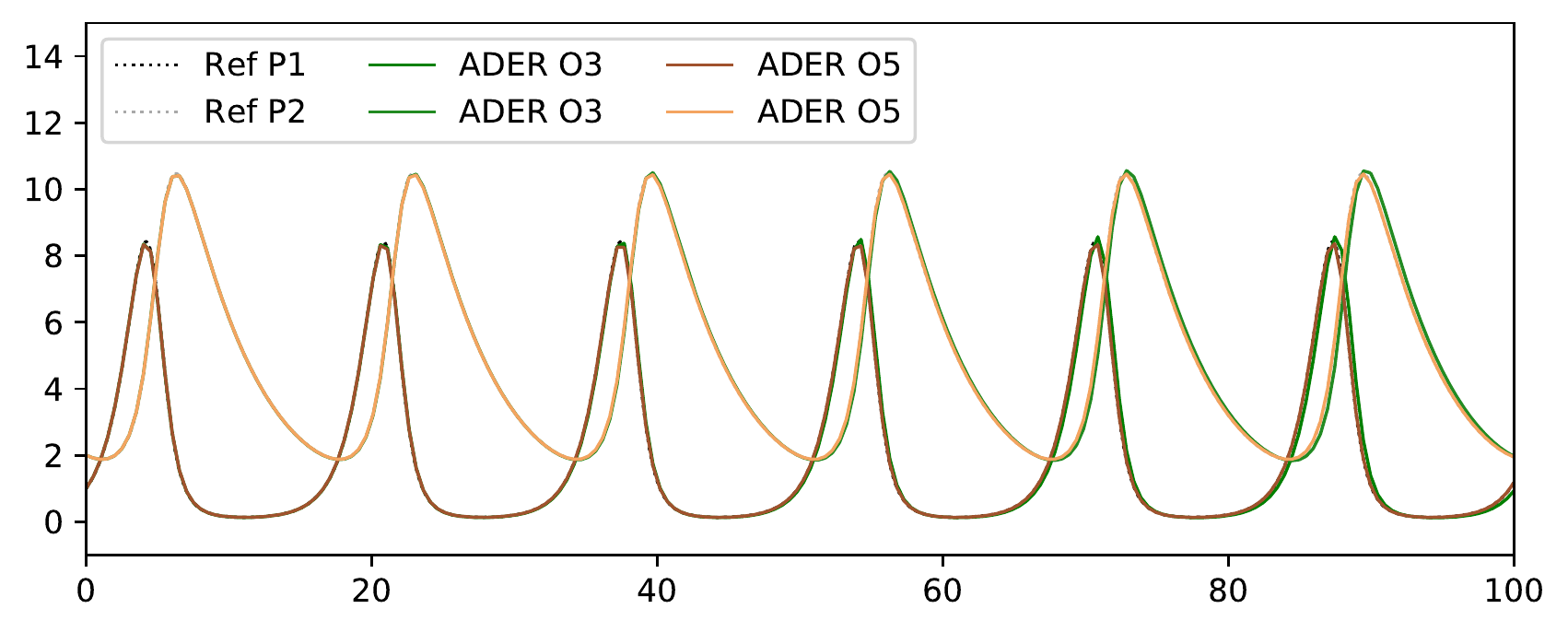}
\caption{Numerical solution of the Lotka-Volterra system \eqref{eq:system-nonlinear}, using initial conditions \eqref{eq:system-nonlinear-ic} using ADER with Gauss--Lobatto nodes at different orders. The top figure uses a timestep $\Delta t = 1$, the bottom figure $\Delta t = 0.5$. \label{fig:lodka-sol-ader}}
	\end{center}
\end{figure}
\begin{figure}
\begin{center}
\includegraphics[width=0.65\linewidth]{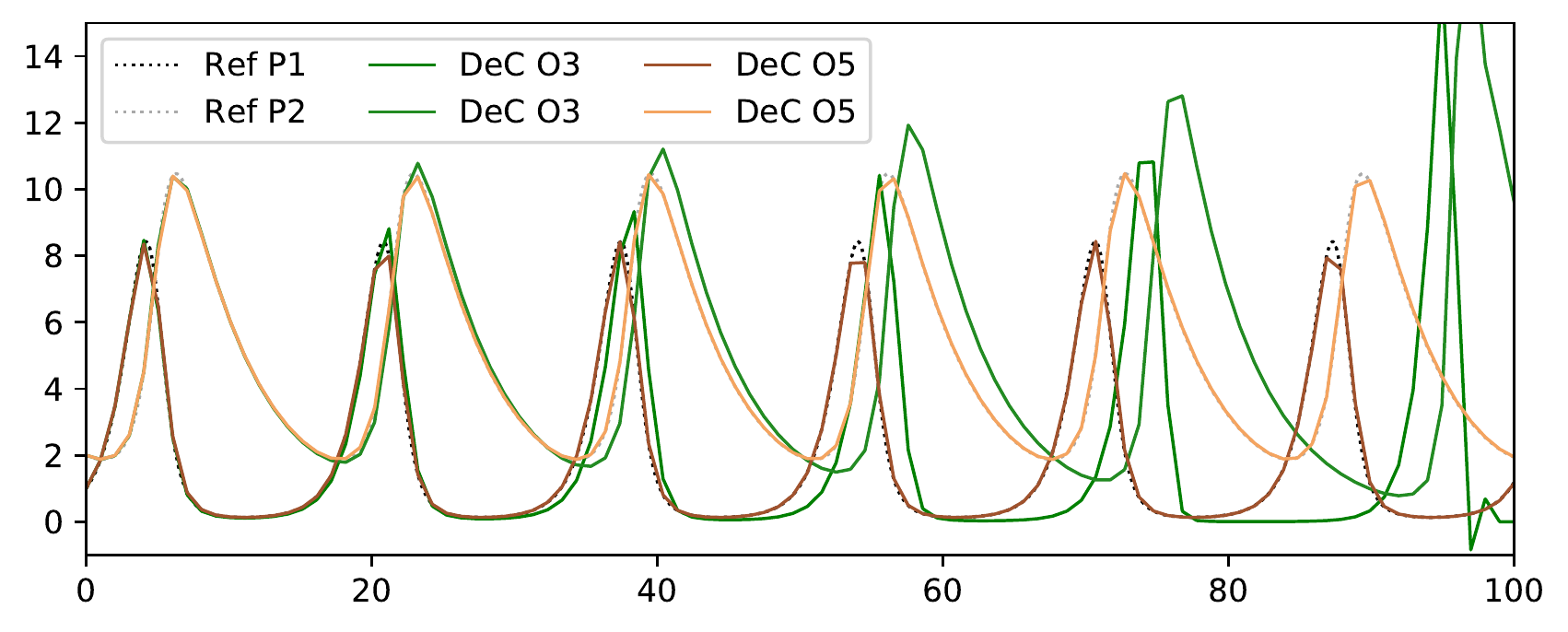}\\
\includegraphics[width=0.65\linewidth]{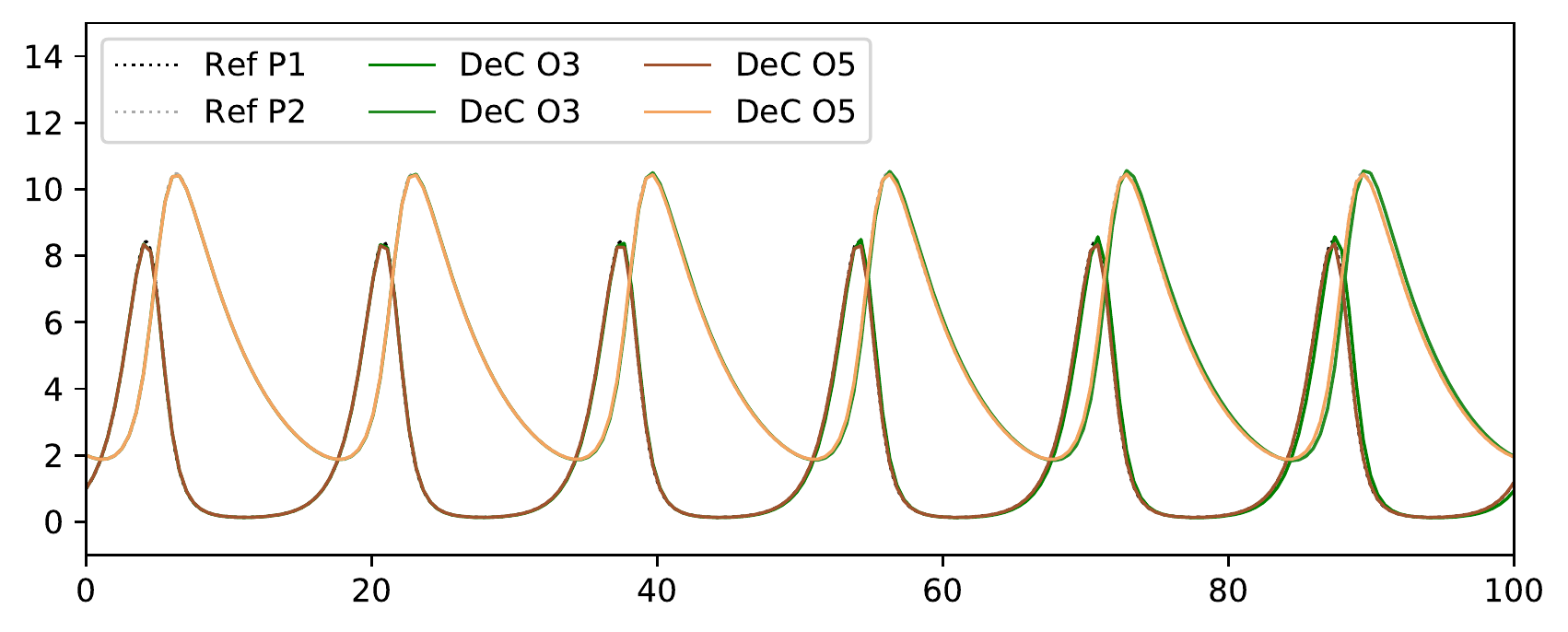}
\caption{Numerical solution of the Lotka-Volterra system \eqref{eq:system-nonlinear}, using initial conditions \eqref{eq:system-nonlinear-ic} using DEC with Gauss-Lobatto nodes at different orders. The top figure uses a timestep $\Delta t = 1$, the bottom figure $\Delta t = 0.5$. \label{fig:lodka-sol-dec}}
\end{center}
\end{figure}

\subsection{Stiff Problems}\label{sec:StiffSimulations}
In this section, we test the implicit methods for few stiff problems.
To start, we consider again the linear system \eqref{eq:system-linear} with a final time of $T=10$ and larger time steps.
For this problem, in the implicit methods, we considered the whole right-hand-side as the stiff part, and, hence, $S$ is defined as
\begin{equation}
	S(y) = \left(-y_1+3y_2, -3y_2+5y_3, -5y_3\right)^T.
\end{equation}
We notice in figure \ref{fig:stiff_linear_system} that the implicit methods resemble the profile of the exact solution even with few timesteps, while the explicit ADER (as well as the explicit DeC not reported in picture) fails to capture the right behavior.

\begin{figure}[h]
\begin{minipage}{0.48\textwidth}
	\includegraphics[width=\linewidth]{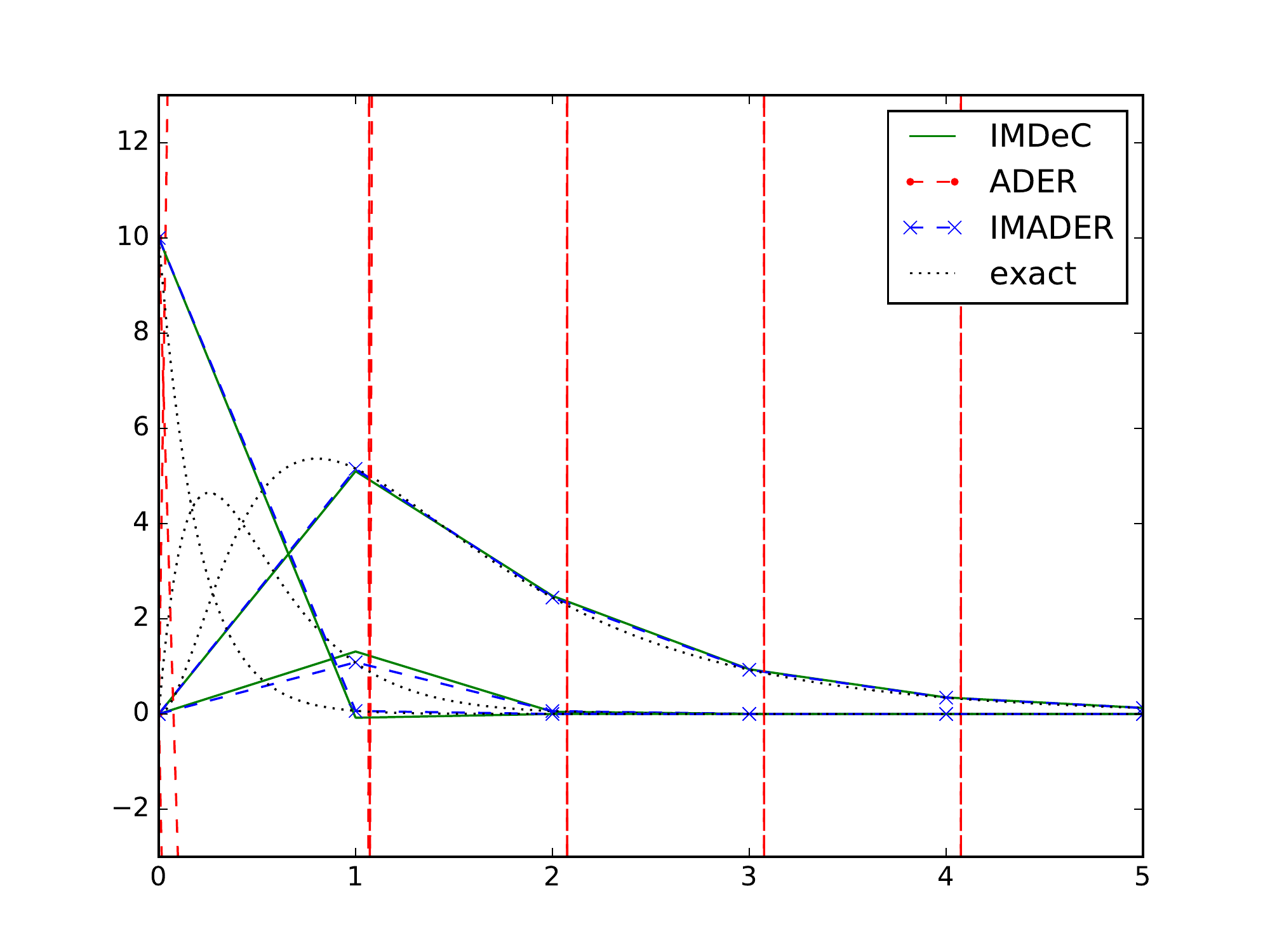}
	\caption{Simulation of \eqref{eq:system-linear} with 4th order implicit DeC, explicit and implicit ADER on Gauss Lobatto points with $\Delta t=1$.}
	\label{fig:stiff_linear_system}
\end{minipage}\hfill
\begin{minipage}{0.48\textwidth}
	\includegraphics[width=\linewidth]{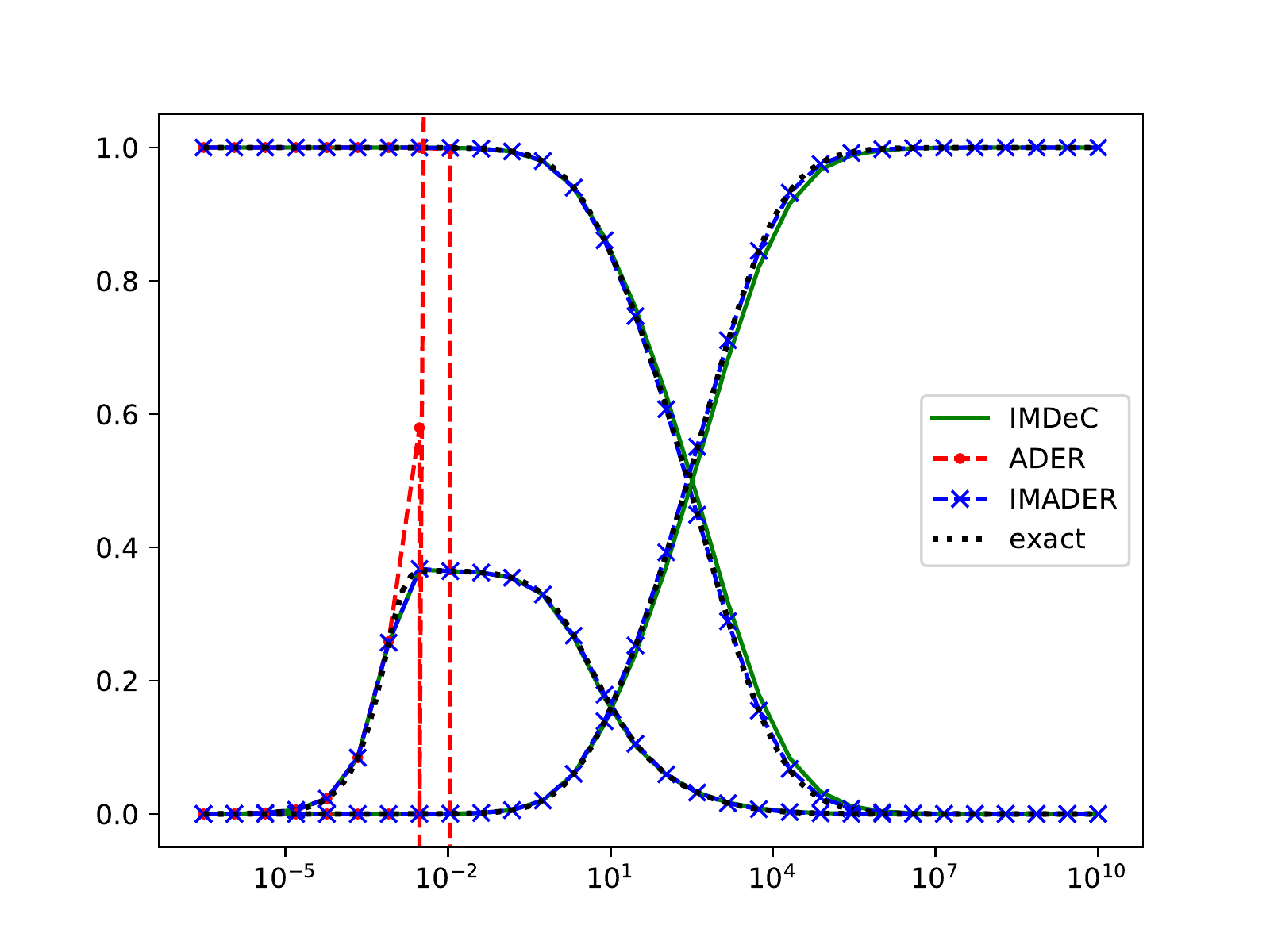}
\caption{Simulation for Robertson problem \eqref{eq:Robertson} with third order implicit DeC, explicit and implicit ADER with 30 timesteps. \label{fig:Robertson}}
\end{minipage}
\end{figure}

The second stiff test we perform is the Robertson problem, which is a highly stiff benchmark problem in the field.
It describes a chemical reaction of 3 constituents and it evolves with different time scales: very rapidly at the beginning and very slowly at the end. 
It is defined by
\begin{equation}\label{eq:Robertson}
	\begin{pmatrix}
	y_1(t)\\y_2(t)\\y_3(t)
	\end{pmatrix}'=\begin{pmatrix}
	10^4y_2(t)y_3(t)-0.04y_1(t)\\  0.04y_1(t) -10^4 y_2(t)y_3(t) -3\cdot 10^7 y_2(t)^2 \\ 3\cdot 10^7 y_2(t)^2
	\end{pmatrix}, \, t\in [0,10^{11}],\qquad y(0)=\begin{pmatrix}
	1\\0\\0
	\end{pmatrix}.
\end{equation}
In this system we consider again in the stiff part the whole right-hand-side, namely $S(y)=F(y)$.
In order to catch the behavior of this simulation the timesteps are chosen progressively increasing ($\Delta t^n=2\Delta t^{n-1}$) as prescribed in \cite{offner2019arbitrary}.
In figure \ref{fig:Robertson} we present the results of the 3 components for the different methods. The second constituent $y_2$ is multiplied by $10^4$ to make its evolution visible on the plot. We observe that the implicit methods are the only ones able to converge to the exact solution and, despite very similar results, the implicit ADER method is more precise than the implicit DeC, as remarked in section \ref{sub_stiff}. 
The explicit methods fail already at the beginning of the simulations because of the wide oscillations.

\subsubsection{PDE Case}
For completeness, we verify our theoretical results focusing on two simple hyperbolic problems.
First, we  consider the linear advection equation, a one dimensional scalar, linear partial differential equation, 
given by the following initial value problem

\begin{equation}\label{eq:advection}
\begin{split}
&\partial_t u + \partial_x u = 0, \quad x\in\mathbb{R}, \quad t>0 \\
&u(x,0) = u_0(x), \quad x \in \mathbb{R}.
\end{split}
\end{equation}

We set for the initial condition
\[u_0(x) = \sin(2\pi x) \quad x \in [0,1].\]

We discretize the equation with the method of lines. The space operator is discretized with the SD method in Section \ref{sec:SD}, whereas the time integration is performed with explicit ADER and DeC (using the different collocation points for the subtimesteps). 

The time step is given as in \cite{Vanharen2017} for the SD method, i.e.,
\[ \Delta t = \frac{C}{M+1} \frac{\Delta x}{|v_{max}|},\]
where $C$ is the Courant factor, set to $0.5$ and $M$ the maximal polynomial degree for the spatial approximation.

In figures \ref{fig:advection-conv-ader} and \ref{fig:advection-conv-dec}, we observe the appropriate convergence rates at $T=1$.
\begin{figure}
\begin{minipage}[c]{0.45\linewidth}
\includegraphics[width=\linewidth]{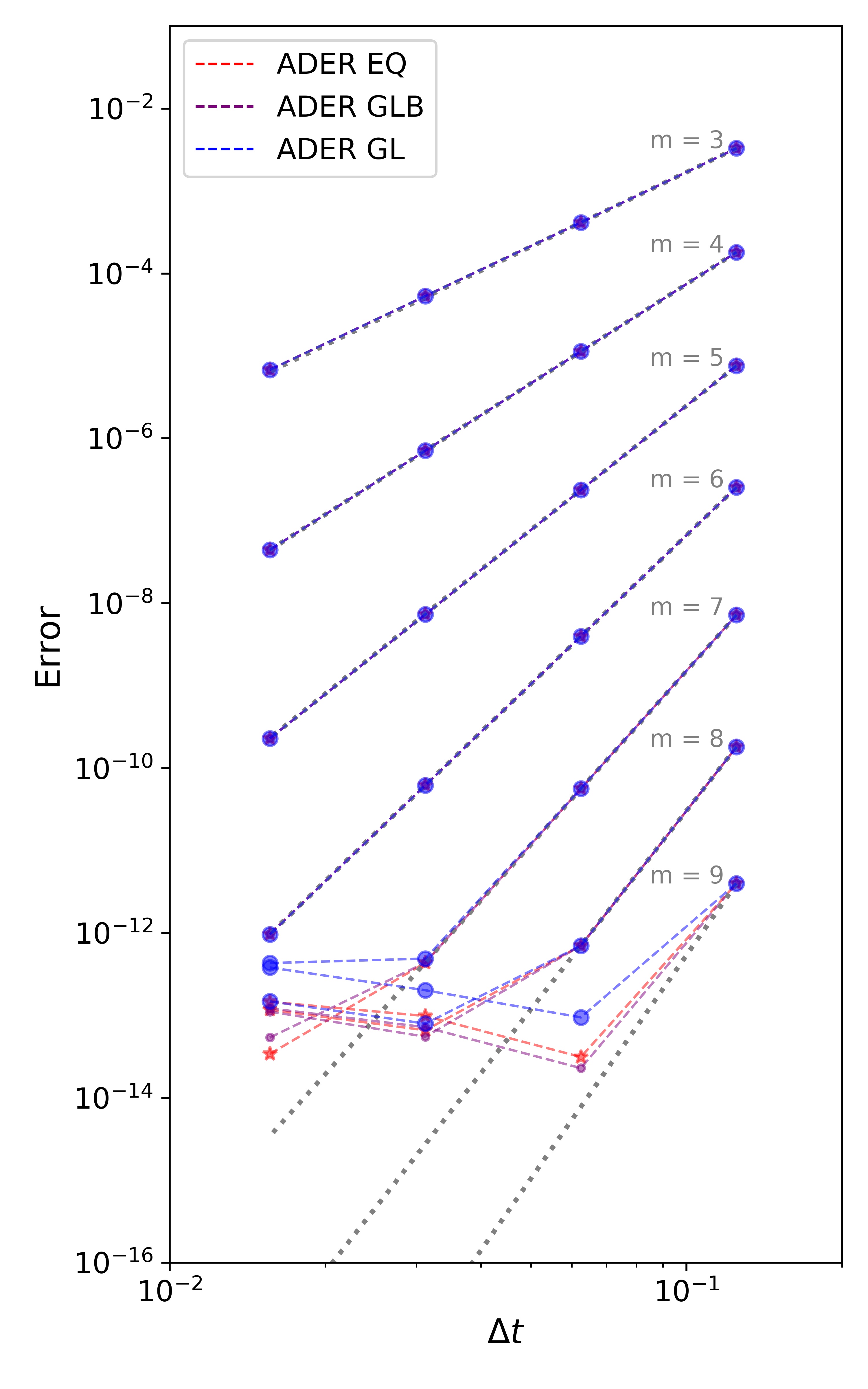}
\caption{Convergence curves for ADER, varying the approximation order and collocation of 
nodes for the subtimesteps for the linear advection equation  \eqref{eq:advection}.}
\label{fig:advection-conv-ader}
\end{minipage}
\hfill
\begin{minipage}[c]{0.45\linewidth}
\includegraphics[width=\linewidth]{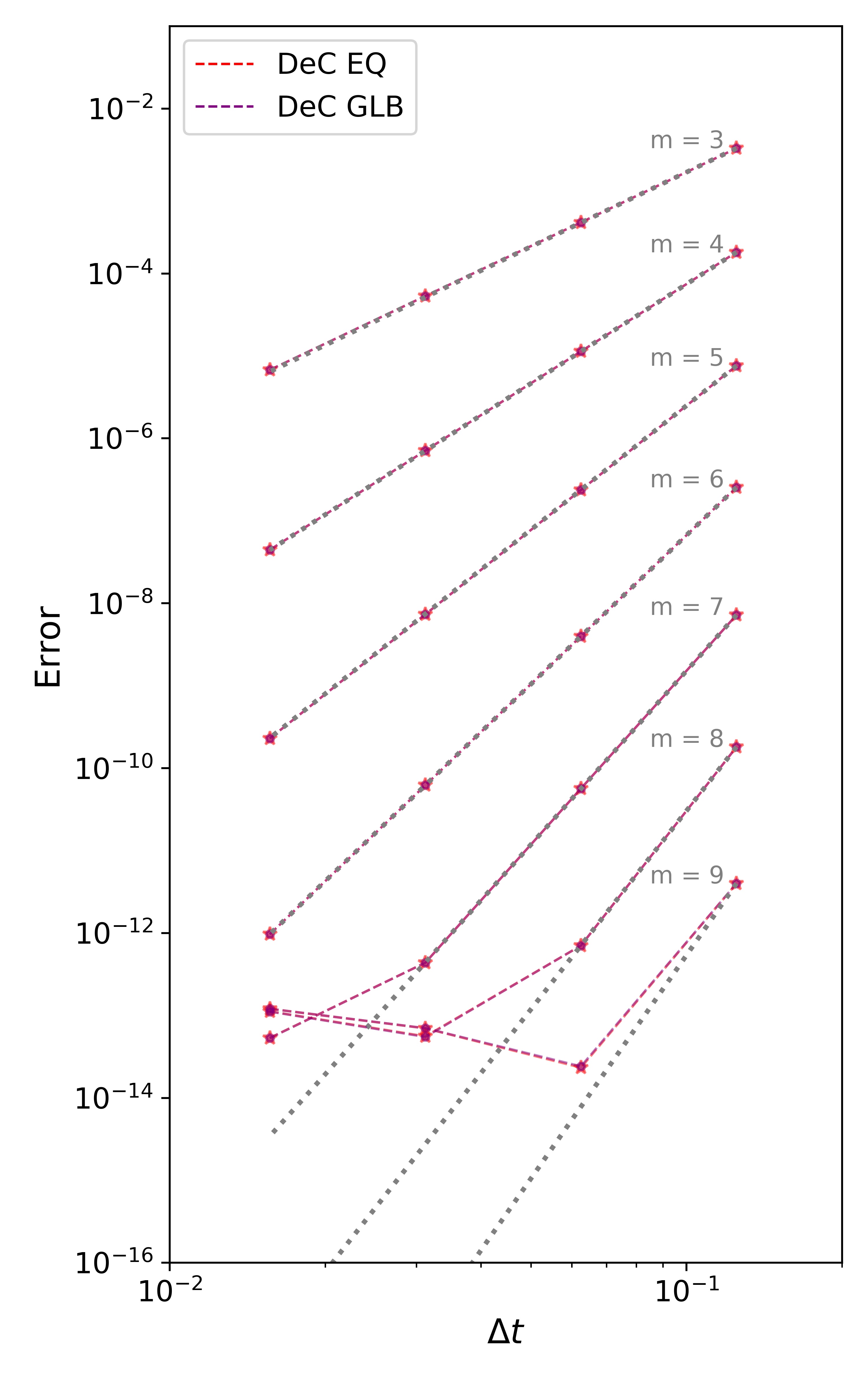}
\caption{
Convergence curves for DeC, varying the approximation order and collocation of 
nodes for the subtimesteps for the linear advection equation  \eqref{eq:advection}.}
\label{fig:advection-conv-dec}
\end{minipage}%
\end{figure}

Finally, we consider the Burgers' equation, a one dimensional scalar, 
nonlinear partial differential equation, given by the initial value problem

\begin{equation}\label{eq:Burgers}
\begin{split}
&\partial_t u + \partial_x \left( \frac{u^2}{2} \right) = 0, \quad x\in\mathbb{R}, \quad t>0 \\
&u(x,0) = u_0(x), \quad x \in \mathbb{R}.
\end{split}
\end{equation}

We consider the initial conditions

\[u_0(x) = \tanh(10x-5) \quad x \in [0,1],\]
with $t \in [0, 0.05]$.

Again, the space operator is discretized with an SD method,
whereas the time integration is performed with ADER and DeC (using the different collocation points for the subtimesteps).
From figures \ref{fig:burgers-conv-ader} and  \ref{fig:burgers-conv-dec}, 
we observe the appropriate convergence rates, with no noticeable differences between the different collocation points. The reference solution is a high resolution numerical solution, interpolated at the necessary spatial points to compute the error. It is to note that the solution in the ADER case was more stable when the update was performed in the following way:
\[  u(t^{n+1}) = u(t^{n}) - \Delta t \partial_x \int_{T^n}  f(\vec{\phi(t)}^T \bbc^{(k+1)}), \]
rather than
\[ u(t^{n+1}) = \vec{\phi(1)}^T \bbc^{(k+1)}. \]

This made no significant difference for the linear advection case.
Extensions to more complex problems including stiff source terms and a stability analysis will be considered in future research.
\begin{figure}
\begin{minipage}[c]{0.45\linewidth}
\includegraphics[width=\linewidth]{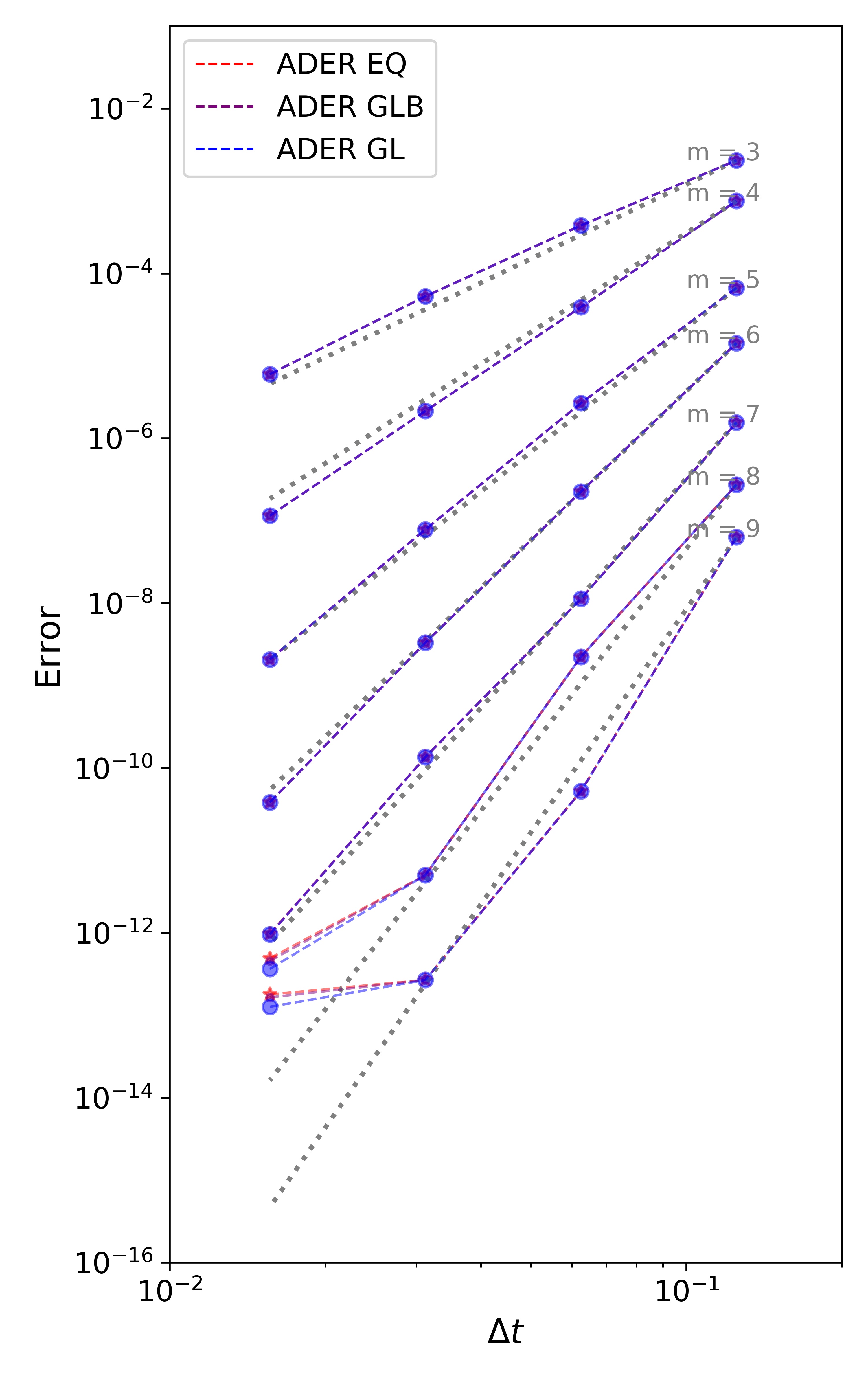}
\caption{Convergence curves for ADER  using SD for the space discretization for the Burgers' Equation \eqref{eq:Burgers}.}
\label{fig:burgers-conv-ader}
\end{minipage}
\hfill
\begin{minipage}[c]{0.45\linewidth}
\includegraphics[width=\linewidth]{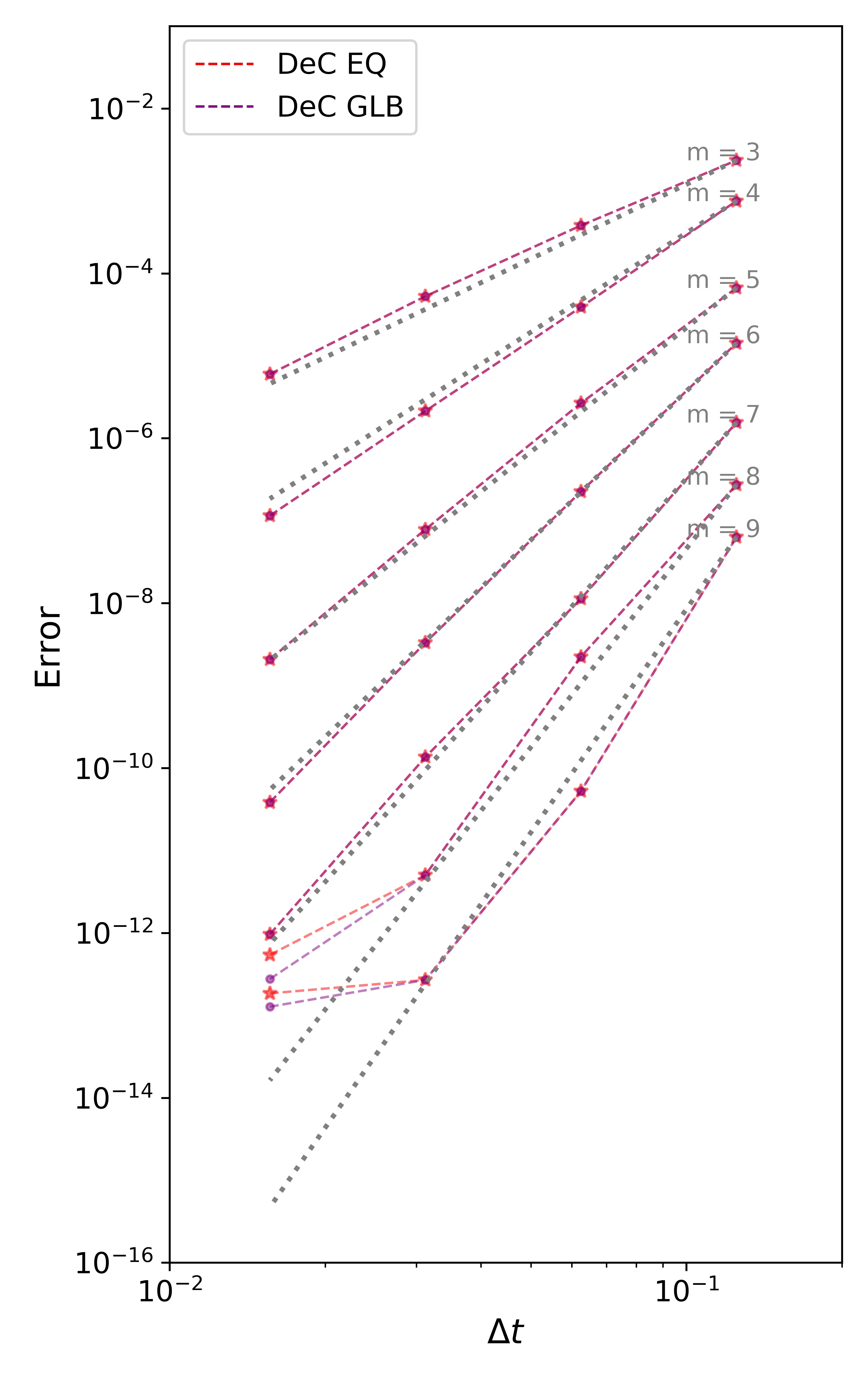}
\caption{Convergence curves for DeC using SD for the space discretization for the Burgers' Equation \eqref{eq:Burgers}.}
\label{fig:burgers-conv-dec}
\end{minipage}%
\end{figure}

\section{Summary and Outlook}\label{se:Summary}

In this paper, we have demonstrated the connection between the DeC framework and the ADER approach. In particular, for the explicit case, we showed that ADER can be interpreted as DeC, as well as DeC is equivalent to ADER up to a choice of test functions.

Since we embed ADER in the theoretical DeC framework, we were able to demonstrate theoretical results for ADER, e.g. how many iterations are needed to obtain the desired order, extending the results of \cite{hjackson2017}.

When considering the implicit versions of DeC and ADER (IMDeC and IMADER, respectively), more apparent differences emerged between these two methods. In particular, although the two methods might appear similar in their formulation, there are major differences: in the IMADER procedure, the mass matrix in time is mixed with the Jacobian of the stiff term, leading to a significantly larger mass matrix which has to be inverted at each time step, whereas in the IMDeC formulation, the stiff term is fully absorbed by the $\L^1$ operator, leading to a much simpler mass matrix. On the other side, the simulations of IMADER seems more accurate from our preliminary results.

We verify our theoretical finding by a variety of numerical simulations, with ODEs and PDEs. At the same time, we also studied the influence of the choice of collocation points in time and verified that there was not much of a difference, except that for very high orders.
For some choices of collocation points, namely, for ADER EQ and ADER GL, a decrease of the orders can be recognized 
when we are close to machine precision. The cause for this is not clear, one explanation could be a bad condition number of the mass matrix or other interpolation issues. However, the answer to this will be left open for future research.

Finally, we hope that, with this paper, the modern ADER approach and DeC become clearer to the hyperbolic community and 
that it becomes clear that these methods are very similar. 
Since we have provided some theoretical background 
for ADER as a time--integration scheme, many further extensions are
possible and new questions can be asked: for example, what is 
the relation between ADER and RK methods, as it was already done for DeC in \cite{christlieb2010integral}.
Another possibility is to rewrite ADER in SSP formulation, in the spirit of \cite{liu2008strong}, 
or to build ADER schemes which are positivity preserving and conservative by using a Patankar trick 
\cite{huang2019positivity,offner2019arbitrary}. Extensions 
to relaxation approaches introduced by Ketcheson et al. \cite{Ketcheson2019RelaxationRM,ranocha2019relaxation} are also possible to construct entropy conservative ADER schemes, where this is already work in progress for the DeC approach. 
Finally, a comparison of the IMADER and IMDeC including variable space discretization and  focusing on more complex stiff hyperbolic problems and the stability of such methods will be addressed in the future. 

Finally, in spirit of open science, all codes used to produce this paper are available in \cite{ourrepo}.

\section*{Acknowledgements}
P. \"Offner has been funded by the UZH Postdoc Grant (Number FK-19-104) and the SNF grant (Number 200021\_175784).
Davide Torlo is supported by ITN ModCompShock project funded by the European Union’s Horizon 2020 research and innovation program under the Marie Sklodowska-Curie grant agreement No 642768. M. Han Veiga acknowledges financial support from MIDAS.

\bibliographystyle{abbrv}
\bibliography{literature}

\end{document}